\documentclass[11pt,letterpaper]{amsart}

\usepackage{enumerate}
\usepackage{amscd}
\usepackage{amsmath}
\usepackage{amssymb}
\usepackage{amsthm}    
\usepackage{latexsym}
\usepackage{color} 
\usepackage{graphicx}
\usepackage{url}

\theoremstyle{plain}
\newtheorem{theorem}{Theorem}[section]
\newtheorem{proposition}[theorem]{Proposition} 
\newtheorem{lemma}[theorem]{Lemma} 
\newtheorem{corollary}[theorem]{Corollary} 

\theoremstyle{definition}
\newtheorem{definition}[theorem]{Definition} 
\newtheorem{question}[theorem]{Question} 
\newtheorem{example}[theorem]{Example} 

\theoremstyle{remark}
\newtheorem{remark}[theorem]{Remark} 

\numberwithin{equation}{section}

\DeclareMathOperator{\sgn}{sgn}

\begin{document}

\title[A volume  correspondence]{A volume correspondence between anti-de Sitter space and its boundary}

\author{Lizhao Zhang}
\email{lizhaozhang@alum.mit.edu}



\date{}

\begin{abstract}
Let $\mathbb{H}^{n+1}_1$ be the $(n+1)$-dimensional anti-de Sitter space (AdS),
in this paper we propose to extend $\mathbb{H}^{n+1}_1$ conformally to another copy of 
$\mathbb{H}^{n+1}_1$ by gluing them along the boundary at infinity,
and denote the resulting space by \emph{double anti-de Sitter space} $\mathbb{DH}^{n+1}_1$.
We propose to introduce a volume $V_{n+1}(P)$ (possibly complex valued) on polytopes $P$ in $\mathbb{DH}^{n+1}_1$
whose facets all have non-degenerate metrics (called \emph{good} polytopes), 
and show that it is well defined and invariant under isometry,
including the case that $P$ contains a non-trivial portion of $\partial\mathbb{H}^{n+1}_1$.
For $n$ even, $V_{n+1}(P)$ is shown to be completely
determined by the intersection of $P$ and $\partial\mathbb{H}^{n+1}_1$,
which leads to the following important applications:
it induces a new intrinsic (conformal) \emph{volume} on good polytopes in $\partial\mathbb{H}^{n+1}_1$
that is invariant under conformal transformations of $\partial\mathbb{H}^{n+1}_1$,
and establishes an AdS-CFT type correspondence between the volumes on 
$\mathbb{DH}^{n+1}_1$ and $\partial\mathbb{H}^{n+1}_1$.
\end{abstract}

\maketitle

\section{Introduction}
\label{section_intro}

The purpose of this note is twofold. First, it is about connections between two seemingly far-removed subjects, 
the \emph{anti-de Sitter space} and \emph{polytope}. 
While polytope can be defined in the anti-de Sitter space, 
namely a finite intersection of half-spaces and possibly unbounded,
it is rarely the focus in the study of the anti-de Sitter space,
where generally more ``smooth'' tools like differential geometry are used.
However we show that polytopes (especially the unbounded ones)
are essential to study the boundary at infinity,
particularly for the topic that we are concerned with in this paper, the volume.
Second, the study is integrated with a new space that we are about to introduce later,
the \emph{double anti-de Sitter space}, which is constructed by gluing two copies of
the anti-de Sitter space along their boundaries.
We show that this newly introduced space is also of interest in its own right.

\subsection{Background and motivations}
\label{section_background}

We first introduce the necessary background that motivates this paper.
Let $\mathbb{H}^n$ be the $n$-dimensional hyperbolic space,
using the \emph{hyperboloid model},
it was shown in Zhang~\cite{Zhang:double_hyperbolic} that one can extend $\mathbb{H}^n$ conformally
to a \emph{two}-sheeted hyperboloid by identifying their boundaries at infinity projectively,
with the resulting space homeomorphic to $\mathbb{S}^n$ and denoted by 
\emph{double hyperbolic space} $\mathbb{DH}^n$ (Section~\ref{section_polytope_volume_DH}).
As one of the most crucial features for the construction of $\mathbb{DH}^n$,
the lower sheet (denoted by $\mathbb{H}^n_{-}$) is not isometric to $\mathbb{H}^n$,
and the length element $ds$ on $\mathbb{H}^n_{-}$ is the negative of the length element $ds$ on $\mathbb{H}^n$.
To compute the geodesic between two points in $\mathbb{H}^n$ and $\mathbb{H}^n_{-}$ respectively
across $\partial\mathbb{H}^n$,
it is analogous to integrating $1/x$ in $\mathbb{R}$ from the negative to the positive across the origin,
where it is not integrable by the standard Lebesgue integral,
but complex analysis can be brought in to make sense of the integration.
Conversely, this also suggests that the role of $\mathbb{H}^n_{-}$ 
cannot be replaced with an \emph{exact} copy of $\mathbb{H}^n$.

A \emph{half-space} in $\mathbb{DH}^n$ is obtained by gluing a half-space in $\mathbb{H}^n$ 
and its antipodal image in $\mathbb{H}^n_{-}$ along $\partial\mathbb{H}^n$ 
by identifying their \emph{opposite} ends projectively.
Though it seems strange that a half-space in $\mathbb{DH}^n$ appears to be on both sides 
of a hyperplane when using the hyperboloid model,
this apparent paradox can be resolved by the fact that $\mathbb{H}^n$ and $\mathbb{H}^n_{-}$
are embedded in two different Minkowski spaces respectively. 
The construction of both $\mathbb{DH}^n$ and its half-spaces, at first sight, 
perhaps looks somewhat unconventional.
But by using the \emph{upper half-space model} and the \emph{hemisphere model} of $\mathbb{H}^n$,
the construction of $\mathbb{DH}^n$ naturally extends both of them across the boundary at $x_0=0$
conformally to the \emph{lower} half-space and the \emph{lower} hemisphere respectively, 
making them more like a ``full-space model'' and a ``full-sphere model'' in some sense.
In the upper half-space model (for the full $\mathbb{DH}^n$), 
the corresponding half-space of $\mathbb{DH}^n$ is either the inside or the outside of a ball centered on $x_0=0$,
or a Euclidean half-space whose face is vertical to $x_0=0$, which appears much more natural than in the hyperboloid model.

A polytope in $\mathbb{H}^n$, possibly unbounded,
is a finite intersection of half-spaces in $\mathbb{H}^n$.
For an overview of polytopes, see Ziegler~\cite{Ziegler:polytopes}.
Similarly, a \emph{polytope} in $\mathbb{DH}^n$ is a finite intersection of half-spaces in $\mathbb{DH}^n$.
It is always symmetric between $\mathbb{H}^n$ and $\mathbb{H}^n_{-}$ through antipodal points,
but may not be homeomorphic to a ball and may possibly contain more than one connected component.
In fact, because of the inclusion of $\mathbb{H}^n_{-}$, 
there are no such notion of \emph{convex} polytope in $\mathbb{DH}^n$.
For any polytope $P$ in $\mathbb{DH}^n$, a volume $V_n(P)$ is introduced,
including the case that the polytope contains a non-trivial portion of $\partial\mathbb{H}^n$.
We remark that the choice to extend the volume on $\mathbb{H}^n$ to $\mathbb{DH}^n$ is not unique,
just like the integration of $1/x$ in $\mathbb{R}$ from the negative to the positive is not unique,
but with a proper choice, the volume is shown to be well defined and invariant under isometry.
The total volume of $\mathbb{DH}^n$ is
\begin{equation}
\label{equation_DH_total_volume}
V_n(\mathbb{DH}^n)=i^n V_n(\mathbb{S}^n)
\end{equation}
for both even and odd dimensions \cite{Zhang:double_hyperbolic},
where $V_n(\mathbb{S}^n)$ is the $n$-dimensional volume of the standard unit $n$-sphere $\mathbb{S}^n$.

Particularly for $n$ odd, the volume $V_n(P)$ is shown to be completely determined 
by the intersection of $P$ and $\partial\mathbb{H}^n$. As an important application,
it induces a new intrinsic (conformal) \emph{volume} on polytopes in $\partial\mathbb{H}^n$
(or more precisely, a real-valued finitely additive measure on $\partial\mathbb{H}^n$,
but with the values not necessarily non-negative)
that is invariant under M\"{o}bius transformations,
i.e., global conformal transformations of $\partial\mathbb{H}^n$ induced by the isometries of $\mathbb{H}^n$.
This is a very strong property for $\partial\mathbb{H}^n$, as most known volumes 
(e.g., the round metric on a sphere) do not have this conformal invariance property.
The volume on $\partial\mathbb{H}^n$ comes from an entirely different mechanism
than the ``usual'' Riemannian metric, and is not induced by any volume form on a differentiable manifold.
In fact, the discovery of the volume on $\partial\mathbb{H}^n$ for $n$ odd is largely due to the fact
that the volume on $\mathbb{DH}^n$ is also introduced on polytopes in the first place.
Whether the volume on $\partial\mathbb{H}^n$ can be defined on a larger class of regions
remains of interest for future research.
We note that $\partial\mathbb{H}^n$ has a much larger class of polytopes than the sphere with round metric,
and by the volume on polytopes $G$ in $\partial\mathbb{H}^n$ (denoted by $V_{\infty,n-1}(G)$),
it unveils that $\partial\mathbb{H}^n$ has hidden geometric properties of the spherical, 
(double) hyperbolic, and Euclidean spaces at the same time \cite[Theorem~12.9]{Zhang:double_hyperbolic}.
Namely, let $n=2m+1$, then for any polytope $G$ with \emph{finite} volume in $M^{2m}$
($\mathbb{S}^{2m}$, $\mathbb{H}^{2m}$, $\mathbb{DH}^{2m}$, or $\mathbb{R}^{2m}$) 
with constant curvature $\kappa$, 
not only $G$ can be treated as a polytope in $\partial\mathbb{H}^{2m+1}$ 
when $M^{2m}$ is treated conformally as $\partial\mathbb{H}^{2m+1}$
(for $\mathbb{H}^{2m}$, treated as ``half'' of $\partial\mathbb{H}^{2m+1}$;
for $\mathbb{R}^{2m}$, treated as $\partial\mathbb{H}^{2m+1}$ with a point removed),
we also have
\begin{equation}
\label{equation_polytope_volume_infinity}
V_{\infty,2m}(G)=\kappa^m V_{2m}(G).
\end{equation}
But somewhat surprisingly, this identity is not true if $G$ is not a polytope in $M^{2m}$.

\subsection{Double anti-de Sitter space}

In this note we show that a similar theory can be developed 
for the \emph{anti-de Sitter space} and its boundary as well.
By convention, let the anti-de Sitter space be $(n+1)$-dimensional and denoted by $\mathbb{H}^{n+1}_1$.
In the case of $n=0$, $\mathbb{H}^1_1$ is a circle and does not have a boundary, 
which we consider trivial and from now on we assume $n\ge 1$.
If we treat $\mathbb{H}^{n+1}_1$ as isometrically embedded in $\mathbb{R}^{n,2}$,
denote by $\mathbb{H}^{n+1}_{1,-}$ (embedded in a different linear space $\mathbb{R}^{n,2}_{-}$,
see Section~\ref{section_preliminaries}) a copy of $\mathbb{H}^{n+1}_1$.
Similar to the role of $\mathbb{H}^n_{-}$ played in the construction of $\mathbb{DH}^n$,
the length element $ds$ on $\mathbb{H}^{n+1}_{1,-}$ is the negative
of the length element $ds$ on $\mathbb{H}^{n+1}_1$.

\begin{definition}
By gluing $\mathbb{H}^{n+1}_1$ to $\mathbb{H}^{n+1}_{1,-}$ along the boundary at infinity 
$\partial\mathbb{H}^{n+1}_1$ by identifying their \emph{opposite} ends projectively,
we obtain a new space and denote it by \emph{double anti-de Sitter space} $\mathbb{DH}^{n+1}_1$
(see Figure~\ref{figure_double_AdS}).
\end{definition}

\begin{figure}[h]
\centering
\resizebox{.4\textwidth}{!}
  {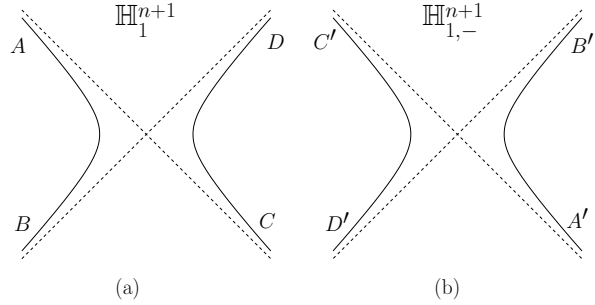}
\caption{The double anti-de Sitter space $\mathbb{DH}^{n+1}_1$ 
is obtained by gluing $\mathbb{H}^{n+1}_1$ to $\mathbb{H}^{n+1}_{1,-}$ 
by identifying their opposite ends projectively, e.g., $A$ with $A'$,  and $B$ with $B'$, etc.}
\label{figure_double_AdS}
\end{figure}

Again, the construction of $\mathbb{DH}^{n+1}_1$ perhaps appears unconventional.
But if the reader can somehow be convinced that the construction of $\mathbb{DH}^n$ 
is the \emph{right} approach to extend $\mathbb{H}^n$ \emph{conformally} across $\partial\mathbb{H}^n$,
then one may in fact find the construction of $\mathbb{DH}^{n+1}_1$ quite natural,
as the two constructions are essentially the same.
We use the notation $\mathbb{DH}^{n+1}_1$
to refer to the double anti-de Sitter space in the general sense,
independent of the model used.
But when the context is clear, without introducing more notations,
for convenience we also use $\mathbb{DH}^{n+1}_1$ to refer to this particular model above,
which is analogous to the hyperboloid model of hyperbolic space.
We introduce some new notions of $\mathbb{DH}^{n+1}_1$ using this model,
but the notions can also be easily extended to other models later.
An \emph{isometry} of $\mathbb{DH}^{n+1}_1$ is an isometry of $\mathbb{H}^{n+1}_1$  
that also preserves the antipodal points in $\mathbb{H}^{n+1}_{1,-}$
(we remark that there are also antipodal points within $\mathbb{H}^{n+1}_1$ itself),
so it is completely determined by the isometry of $\mathbb{H}^{n+1}_1$.
In $\mathbb{R}^{n,2}$, any hyperplane containing the origin cuts $\mathbb{H}^{n+1}_1$ into two half-spaces.

\begin{definition}
\label{definition_half_space}
A \emph{half-space} in $\mathbb{DH}^{n+1}_1$ is obtained by gluing a half-space in $\mathbb{H}^{n+1}_1$ 
and its antipodal image in $\mathbb{H}^{n+1}_{1,-}$ along $\partial\mathbb{H}^{n+1}_1$ 
by identifying their \emph{opposite} ends projectively,
and a \emph{polytope} in $\mathbb{DH}^{n+1}_1$ is a finite intersection of half-spaces in $\mathbb{DH}^{n+1}_1$.
\end{definition}

A polytope in $\mathbb{DH}^{n+1}_1$ is always symmetric between 
$\mathbb{H}^{n+1}_1$ and $\mathbb{H}^{n+1}_{1,-}$ through antipodal points.
We note that while both $\mathbb{H}^{n+1}_1$ and $\mathbb{H}^{n+1}_{1,-}$ are ``half'' of $\mathbb{DH}^{n+1}_1$, 
by definition they are not half-spaces in $\mathbb{DH}^{n+1}_1$.
Similarly, a polytope in $\mathbb{H}^{n+1}_1$ 
(a finite intersection of half-spaces in $\mathbb{H}^{n+1}_1$ and possibly unbounded)
by itself is not a polytope in $\mathbb{DH}^{n+1}_1$.
We remark that a polytope in $\mathbb{DH}^{n+1}_1$ may not be homeomorphic to a ball,
and can also possibly contain more than one connected component,
e.g., one in $\mathbb{H}^{n+1}_1$ and one in $\mathbb{H}^{n+1}_{1,-}$ respectively.

Similar to $\mathbb{DH}^n$, we want to introduce a volume $V_{n+1}(P)$ on polytopes in $\mathbb{DH}^{n+1}_1$.
But unlike $\mathbb{DH}^n$, in $\mathbb{DH}^{n+1}_1$ the face of a half-space may have degenerate metric,
and we will show that if a polytope $P$ in $\mathbb{DH}^{n+1}_1$ 
contains a facet with degenerate metric, then $V_{n+1}(P)$ may not exist.
For this reason, we introduce the following notion.

\begin{definition}
\label{definition_algebra_half_space}
A \emph{good half-space} in $\mathbb{DH}^{n+1}_1$ is a half-space whose face has non-degenerate metric,
and a \emph{good polytope} in $\mathbb{DH}^{n+1}_1$
is a finite intersection of good half-spaces in $\mathbb{DH}^{n+1}_1$.
Let $\mathcal{H}$ (resp. $\mathcal{H}_0$) be the algebra over $\mathbb{DH}^{n+1}_1$ 
generated by half-spaces (resp. good half-spaces) in $\mathbb{DH}^{n+1}_1$.
\end{definition}

We caution that while by definition all facets of a good polytope have non-degenerate metrics,
they may not all be lower $n$-dimensional \emph{good} polytopes.
This is because a codimension 2 face of a good polytope may still have degenerate metric,
and as a result the facet containing this codimension 2 face is not a lower $n$-dimensional good polytope.
However, this degeneracy is not a concern for our results.

\subsection{Main results}

One of the main goals of this paper is to properly
extend the volume on $\mathbb{H}^{n+1}_1$ to $\mathbb{DH}^{n+1}_1$,
such that it is also compatible with the volume elements of both 
$\mathbb{H}^{n+1}_1$ and $\mathbb{H}^{n+1}_{1,-}$.
To build the theory for $\mathbb{DH}^{n+1}_1$,
we follow a similar methodology as in Zhang~\cite{Zhang:double_hyperbolic}
where $\mathbb{DH}^n$ was introduced,
We introduce a volume $V_{n+1}(P)$ on good polytopes in $\mathbb{DH}^{n+1}_1$.

\begin{theorem}
\label{theorem_polytope_volume_finite_invariant}
Let $P\in\mathcal{H}_0$ in $\mathbb{DH}^{n+1}_1$,
then $V_{n+1}(P)$ is well defined and invariant under isometry.
\end{theorem}

By the argument above that $V_{n+1}(P)$ may not exist
if a polytope $P$ contains a facet with degenerate metric, 
Theorem~\ref{theorem_polytope_volume_finite_invariant} 
cannot be strengthened by replacing $P\in\mathcal{H}_0$ with $P\in\mathcal{H}$.
Similar to $\mathbb{DH}^n$, the choice to define $V_{n+1}(P)$ in $\mathbb{DH}^{n+1}_1$ is not unique.
But unlike $\mathbb{DH}^n$, the Lorentzian metric of $\mathbb{DH}^{n+1}_1$ makes some issues, 
including but not limited to convergence issues,
more difficult to handle and require more techniques to make them work.
We also remark that $V_{n+1}(P)$ is only finitely but not countably additive
(see Example~\ref{example_countably_additive_AdS}).

\begin{theorem}
\label{theorem_volume_AdS_real_imaginary}
Let $P\in\mathcal{H}_0$ in $\mathbb{DH}^{n+1}_1$,
then $V_{n+1}(P)$ is real for $n$ odd,
and $V_{n+1}(P)$ is imaginary for $n$ even and is completely determined by 
the intersection of $P$ and $\partial\mathbb{H}^{n+1}_1$.
\end{theorem}

\begin{remark}
We remark that while a bounded polytope in $\mathbb{H}^{n+1}_1$ always has finite \emph{real} volume
for both odd and even dimensions, by definition it is not an element in $\mathcal{H}_0$,
so this property does not contradict Theorem~\ref{theorem_volume_AdS_real_imaginary}.
For $n$ even, Theorem~\ref{theorem_volume_AdS_real_imaginary} implies that 
the information of $V_{n+1}(P)$ is completely encoded in the boundary at infinity $\partial\mathbb{H}^{n+1}_1$.
As an important application of Theorem~\ref{theorem_volume_AdS_real_imaginary},
for $n$ even and $n\ge 2$, the volume on $\mathbb{DH}^{n+1}_1$ induces an intrinsic (conformal) \emph{volume}
on $\partial\mathbb{H}^{n+1}_1$ that is invariant under conformal transformations of 
$\partial\mathbb{H}^{n+1}_1$ (see below).
Another main application is that for $n$ even, it establishes an AdS-CFT type correspondence 
between the volumes on $\mathbb{DH}^{n+1}_1$ and $\partial\mathbb{H}^{n+1}_1$.
What makes the construction of $\mathbb{DH}^{n+1}_1$ so special 
or in some aspects even necessary is that, to have a correct setup,
technically it is the volume on $\mathbb{DH}^{n+1}_1$ instead of the volume on $\mathbb{H}^{n+1}_1$, 
that makes sense of this AdS-CFT type correspondence to 
the volume on $\partial\mathbb{H}^{n+1}_1$.
\end{remark}

For a good half-space in $\mathbb{DH}^{n+1}_1$, its restriction to $\partial\mathbb{H}^{n+1}_1$ 
is also called a \emph{good half-space} in $\partial\mathbb{H}^{n+1}_1$.
An important application of Theorem~\ref{theorem_volume_AdS_real_imaginary} is that for $n$ even,
it induces an intrinsic real-valued finitely additive measure 
(with the values not necessarily non-negative) on $\mathcal{F}_0$,
the algebra generated by good half-spaces in $\partial\mathbb{H}^{n+1}_1$.
For any $G\in\mathcal{F}_0$, choose any $P\in\mathcal{H}_0$ such that $G=P\cap \partial\mathbb{H}^{n+1}_1$,
by assigning $V_{n+1}(P)$ to $G$ and adjust by a constant factor later,
we call it the \emph{volume} of $G$ and denote it by $V_{\infty,n}(G)$.

\begin{theorem}
\label{theorem_volume_boundary_at_infinity}
For $n$ even and $n\ge 2$, let $G\in\mathcal{F}_0$ in $\partial\mathbb{H}^{n+1}_1$,
then $V_{\infty,n}(G)$ is well defined and invariant under conformal transformations
of $\partial\mathbb{H}^{n+1}_1$.
\end{theorem}

In this paper the \emph{conformal transformations} of $\partial\mathbb{H}^{n+1}_1$
always refer to the \emph{global} conformal transformations of 
$\partial\mathbb{H}^{n+1}_1$ induced by the isometries of $\mathbb{H}^{n+1}_1$.
To the best of our knowledge, both the definition of this ``conformal volume'' $V_{\infty,n}(G)$ on $G$ 
in $\partial\mathbb{H}^{n+1}_1$ and its conformal invariance property are new in the literature.

\subsection{Strategy overview}

To extend the volume on $\mathbb{H}^{n+1}_1$ to $\mathbb{DH}^{n+1}_1$,
the integral of the volume element across $\partial\mathbb{H}^{n+1}_1$ cannot be defined
by the standard Lebesgue integral.
To fix this issue, we give a complex perturbation to the volume element of the space,
with the underlying space endowed with complex valued ``Lorentzian metric'' instead,
and define a volume as the integral of the perturbed volume element.
We note that there is no complex geometry involved, 
and contrary to what it may appear at first sight, complex analysis plays only a minimal role here.
In fact, besides some basic understanding of the different models of $\mathbb{H}^{n+1}_1$ and $\mathbb{H}^n$,
no further knowledge of differential geometry including Lorentzian geometry is assumed of the reader.
To the best of our knowledge, our extension of $\mathbb{H}^{n+1}_1$ to $\mathbb{DH}^{n+1}_1$
is the first \emph{conformal} extension made to $\mathbb{H}^{n+1}_1$ in the literature,
which follows a similar methodology employed by \cite{Zhang:double_hyperbolic} 
to extend $\mathbb{H}^n$ conformally to $\mathbb{DH}^n$.
See also Cho and Kim~\cite{ChoKim} for a projective extension of the volume on hyperbolic space to the de Sitter space.
For a combinatorial treatment of polyhedra (of 3-dimensional only)
in a projective extension of the anti-de Sitter space $\mathbb{H}^3_1$, 
e.g., see Chen and Schlenker~\cite{ChenSchlenker}.

The computation of $V_{n+1}(P)$ is performed in a different model of $\mathbb{DH}^{n+1}_1$,
which is based on a well known model of the anti-de Sitter space using the Minkowski space
(see Matsuda~\cite{Matsuda:antidesitter}). Particularly, when using this model, 
we provide a new definition (to our knowledge) of the \emph{inversion} in Minkowski space, 
different from the traditional understanding (see Remark~\ref{remark_Minkowski_inversion}).

\section{Preliminaries}
\label{section_preliminaries}

We recall here some basic notions of the anti-de Sitter space $\mathbb{H}^{n+1}_1$,
and also introduce new notions for the double anti-de Sitter space $\mathbb{DH}^{n+1}_1$.

Let $\mathbb{R}^{n,2}$ be an $(n+2)$-dimensional vector space endowed with a bilinear product
\begin{equation}
\label{equation_bilinear_product}
x\cdot y=x_0y_0+\cdots+x_{n-1}y_{n-1}-x_ny_n-x_{n+1}y_{n+1},
\end{equation}
then the bilinear product on $\mathbb{R}^{n,2}$ induces a metric
$ds^2=dx_0^2+\cdots+dx_{n-1}^2-dx_n^2-dx_{n+1}^2$,
with the length element $ds=(ds^2)^{1/2}$.
An $(n+1)$-dimensional \emph{anti-de Sitter space} $\mathbb{H}^{n+1}_1$ is defined by
\[ \bigl\{x \in \mathbb{R}^{n,2}: x\cdot x = -1 \bigr\}.
\]
The 1-dimensional anti-de Sitter space $\mathbb{H}^{1}_1$,
which satisfies $-x_0^2-x_1^2=-1$ in $\mathbb{R}^{0,2}$, is a circle with timelike geodesics.
For $n\ge 1$, let the \emph{boundary at infinity} $\partial\mathbb{H}^{n+1}_1$
be the end of those \emph{half}-lines that lie on the \emph{light cone}
$\{x \in \mathbb{R}^{n,2}: x\cdot x = 0\}$ in $\mathbb{R}^{n,2}$.
The boundary is homeomorphic to $\mathbb{S}^{n-1}\times\mathbb{S}^1$.

In this paper the length element $ds$ can also be $-(ds^2)^{1/2}$ in some other cases,
which is crucial for the construction of $\mathbb{DH}^{n+1}_1$.
Let $\mathbb{R}^{n,2}_{-}$ be a copy of $\mathbb{R}^{n,2}$
where it is endowed with the same bilinear product (\ref{equation_bilinear_product}),
but the length element $ds$ on $\mathbb{R}^{n,2}_{-}$
is the negative of the $ds$ on $\mathbb{R}^{n,2}$
\[ds=-(dx_0^2+\cdots+dx_{n-1}^2-dx_n^2-dx_{n+1}^2)^{1/2}.
\]
Namely, $\mathbb{R}^{n,2}_{-}$ has the same $ds^2$ as $\mathbb{R}^{n,2}$, 
but is \emph{not} isometric to $\mathbb{R}^{n,2}$.
Define $\mathbb{H}^{n+1}_{1,-}$ by
\[ \bigl\{x \in \mathbb{R}^{n,2}_{-}: x\cdot x = -1 \bigr\}.
\]
Let $\partial\mathbb{H}^{n+1}_{1,-}$ be the end of those half-lines that lie on the other light cone
$\{x \in \mathbb{R}^{n,2}_{-}: x\cdot x = 0\}$ in $\mathbb{R}^{n,2}_{-}$.
By identifying $\partial\mathbb{H}^{n+1}_1$ with the \emph{opposite} ends
of $\partial\mathbb{H}^{n+1}_{1,-}$ projectively, 
we glue $\mathbb{H}^{n+1}_1$ and $\mathbb{H}^{n+1}_{1,-}$ together
and denote the resulting space by \emph{double anti-de Sitter space} $\mathbb{DH}^{n+1}_1$.
Further computation shows that $\mathbb{DH}^{n+1}_1$ is homeomorphic to $\mathbb{S}^n\times\mathbb{S}^1$.

\section{A new model for $\mathbb{DH}^{n+1}_1$}
\label{section_model_R_D}

For $n\ge 1$, Matsuda~\cite{Matsuda:antidesitter} used an upper half-space $U^{n+1}$
to partially represent $\mathbb{H}^{n+1}_1$, more precisely, half of $\mathbb{H}^{n+1}_1$. 
The model is analogous to the upper half-space model of $\mathbb{H}^n$.
In this section, we extend this model to fully represent the double anti-de Sitter space $\mathbb{DH}^{n+1}_1$.
This new model of $\mathbb{DH}^{n+1}_1$
is particularly useful for computing the volume introduced on $\mathbb{DH}^{n+1}_1$.
A similar representation of the \emph{de Sitter space} was given by Nomizu~\cite{Nomizu:Lorentz}.

\subsection{An isometric embedding}
\label{section_isometric_embedding}

Let $\mathbb{R}^{n,1}$ be the $(n+1)$-dimensional Minkowski space endowed with a bilinear product
\begin{equation}
\label{equation_bilinear_product_Minkowski}
x\cdot y=x_0y_0+\cdots+x_{n-1}y_{n-1}-x_ny_n,
\end{equation}
then the bilinear product induces a metric $ds^2=dx_0^2+\cdots+dx_{n-1}^2-dx_n^2$
on $\mathbb{R}^{n,1}$, with the length element $ds=(ds^2)^{1/2}$.
Let $\mathcal{R}$ be an $(n+1)$-dimensional space endowed with the Lorentzian metric
\begin{equation}
\label{equation_Lorentz_metric}
ds=(dx_0^2+\cdots+dx_{n-1}^2-dx_n^2)^{1/2}/x_0
\end{equation}
for both $x_0>0$ and $x_0<0$.
So $\mathcal{R}$ is conformally equivalent to the metric of Minkowski space $\mathbb{R}^{n,1}$ 
(except at $x_0=0$) by a factor of $1/x_0$ (referred as the \emph{conformal factor}).
By convention the volume element of $\mathbb{R}^{n,1}$ is $dx_0\cdots dx_n$,
then the associated volume element of $\mathcal{R}$ for both $x_0>0$ and $x_0<0$,
multiplying $dx_0\cdots dx_n$ by $1/x_0^{n+1}$, is
\begin{equation*}
dx_0\cdots dx_n/x_0^{n+1}.
\end{equation*}
For $n$ even, notice that the coefficient $1/x_0^{n+1}$ is negative for $x_0<0$.
While the conformal factor $1/x_0$ is negative for $x_0<0$ and is not continuous at $x_0=0$, 
it is an \emph{analytic function} of $x_0$, as opposed to other factors like $|1/x_0|$.
This is one of the main reasons that makes it possible for us to use complex analysis 
to introduce a volume across $x_0=0$ in $\mathcal{R}$,
analogous to integrating $1/x$ in $\mathbb{R}$ from the negative to the positive across the origin
by using complex analysis.

Consider the upper half-space and the lower half-space in $\mathcal{R}$
\begin{equation}
\label{equation_upper_lower_half_space}
U^{n+1}=\{x\in \mathcal{R}: x_0>0\}
\quad\text{and}\quad
L^{n+1}=\{x\in \mathcal{R}: x_0<0\},
\end{equation}
Matsuda~\cite{Matsuda:antidesitter} showed that $U^{n+1}$ can be isometrically embedded
into the anti-de Sitter space $\mathbb{H}^{n+1}_1$.
The isometric embedding $f: U^{n+1}\to\mathbb{H}^{n+1}_1$ is defined by:
\[f(x_0,\dots,x_n)=(y_0,\dots,y_{n+1}),
\]
where
\begin{equation}
\label{equation_isometric_embedding}
\begin{cases}
y_0=(1-x^2)/2x_0     \\
y_i=-x_i/x_0,\quad 1\le i\le n    \\
y_{n+1}=(1+x^2)/2x_0,
\end{cases}
\end{equation}
with $x^2$ the bilinear product on $\mathbb{R}^{n,1}$ (\ref{equation_bilinear_product_Minkowski}).
The image $f(U^{n+1})$ is the open submanifold 
$V^{n+1}=\{y\in\mathbb{H}^{n+1}_1:  y_0+y_{n+1} > 0\}$,
which is half of $\mathbb{H}^{n+1}_1$.

Our goal is to extend $U^{n+1}$ to a bigger domain $\mathcal{R}_D$,
such that $f$ can be extended to be a full isometry 
between $\mathcal{R}_D$ and $\mathbb{DH}^{n+1}_1$.
Denote $W^{n+1}=\{y\in\mathbb{H}^{n+1}_1:  y_0+y_{n+1} < 0\}$,
$V^{n+1}_{-}=\{y\in\mathbb{H}^{n+1}_{1,-}:  y_0+y_{n+1} > 0\}$,
and $W^{n+1}_{-}=\{y\in\mathbb{H}^{n+1}_{1,-}:  y_0+y_{n+1} < 0\}$.

The following property is important for the construction.
For the lower half-space $L^{n+1}$ in $\mathcal{R}$ (\ref{equation_upper_lower_half_space}),
in order for $f(L^{n+1})$ (the same $f$ in (\ref{equation_isometric_embedding}))
to be an \emph{isometric} embedding, because the $ds$ on $L^{n+1}$ is $-(ds^2)^{1/2}$
for $x_0<0$ (see (\ref{equation_Lorentz_metric}))
and recall that the $ds$ on $\mathbb{H}^{n+1}_{1,-}$ is $-(ds^2)^{1/2}$,
so $L^{n+1}$ should be mapped into
$\mathbb{H}^{n+1}_{1,-}$ instead of $\mathbb{H}^{n+1}_1$.
The image $f(L^{n+1})$ is $W^{n+1}_{-}$ (see Figure~\ref{figure_mapping_AdS_R_D} (a)).

Let $\mathcal{R}_{-}$ be a copy of $\mathcal{R}$ but endowed with the Lorentzian metric (notice the minus sign)
\begin{equation}
\label{equation_Lorentz_metric_negative}
ds=-(dx_0^2+\cdots+dx_{n-1}^2-dx_n^2)^{1/2}/x_0
\end{equation}
for both $x_0>0$ and $x_0<0$.
The associated volume element of $\mathcal{R}_{-}$, multiplying $dx_0\cdots dx_n$ by $(-1/x_0)^{n+1}$, is 
\[(-1)^{n+1}dx_0\cdots dx_n/x_0^{n+1}.
\]
Similarly, for the upper half-space and the lower half-space in $\mathcal{R}_{-}$
\begin{equation}
\label{equation_upper_lower_half_space_minus}
U^{n+1}_{-}=\{x\in \mathcal{R}_{-}: x_0>0\}
\quad\text{and}\quad
L^{n+1}_{-}=\{x\in \mathcal{R}_{-}: x_0<0\},
\end{equation}
in order for both $f(U^{n+1}_{-})$ and $f(L^{n+1}_{-})$ 
(the same $f$ in (\ref{equation_isometric_embedding})) to be isometric embeddings,
they should be mapped into $\mathbb{H}^{n+1}_{1,-}$ and $\mathbb{H}^{n+1}_1$ respectively,
with $f(U^{n+1}_{-})=V^{n+1}_{-}$ and $f(L^{n+1}_{-})=W^{n+1}$ (see Figure~\ref{figure_mapping_AdS_R_D} (b)).

\begin{figure}[h]
\centering
\resizebox{.5\textwidth}{!}
  {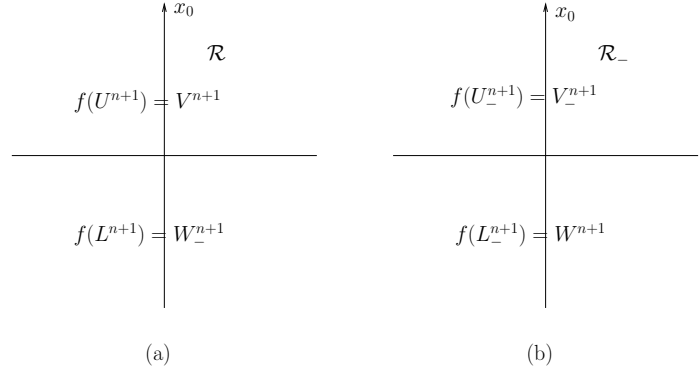}
\caption{An isometric embedding of $\mathcal{R}$ and $\mathcal{R}_{-}$ into $\mathbb{DH}^{n+1}_1$}
\label{figure_mapping_AdS_R_D}
\end{figure}

\subsection{A gluing procedure}

In $\mathcal{R}$ for any point on $x_0=0$, by taking limit from above in the upper half-space $U^{n+1}$,
it is mapped by $f$ into $\partial\mathbb{H}^{n+1}_1$;
and by taking limit from below in the lower half-space $L^{n+1}$,
it is mapped by $f$ into the opposite end on $\partial\mathbb{H}^{n+1}_{1,-}$,
agreeing with the identification $\partial\mathbb{H}^{n+1}_1=\partial\mathbb{H}^{n+1}_{1,-}$
that we set earlier.
In $\mathcal{R}$ by gluing $U^{n+1}$ and $L^{n+1}$ along $x_0=0$,
the whole space $\mathcal{R}$ is mapped into $\mathbb{DH}^{n+1}_1$ as a half-space,
which contains $V^{n+1}$ and $W^{n+1}_{-}$ glued along $\partial\mathbb{H}^{n+1}_1$.
Similarly, in $\mathcal{R}_{-}$ by gluing $U^{n+1}_{-}$ and $L^{n+1}_{-}$ along $x_0=0$,
the whole space $\mathcal{R}_{-}$ is mapped into $\mathbb{DH}^{n+1}_1$ as a half-space,
which contains $V^{n+1}_{-}$ and $W^{n+1}$ glued along $\partial\mathbb{H}^{n+1}_1$.

\begin{remark}
In order for $f(\mathcal{R})$ and $f(\mathcal{R}_{-})$ to be well defined at $x_0=0$, 
in $\mathbb{DH}^{n+1}_1$ it is crucial that $\partial\mathbb{H}^{n+1}_1$ is identified with 
the \emph{opposite} ends of $\partial\mathbb{H}^{n+1}_{1,-}$ projectively.
\end{remark}

To further glue $\mathcal{R}$ and $\mathcal{R}_{-}$ along some ``points at infinity'' $\partial\mathcal{R}$
to form the new model $\mathcal{R}_D$ for $\mathbb{DH}^{n+1}_1$
(the quotation marks are added because $\partial\mathcal{R}$ is not the 
\emph{boundary at infinity} as $\partial\mathbb{H}^{n+1}_1$ is,
but only at some ``infinity'' as a model-specific notion for $\mathcal{R}$ and $\mathcal{R}_{-}$),
and extend $f$ to be a one-to-one mapping to $\mathbb{DH}^{n+1}_1$,
we first glue $U^{n+1}$ and $L^{n+1}_{-}$ together.
We use those null lines $l$ in $\mathcal{R}$, 
who has a \emph{upper} part in $U^{n+1}$ and a \emph{lower} part in $L^{n+1}$, 
to illustrate the gluing procedure. Let $l_{-}$ be $l$'s copy in $\mathcal{R}_{-}$.
By $f$ the upper part of $l$ in $U^{n+1}$ and the lower part of $l_{-}$ in $L^{n+1}_{-}$ are mapped to
a full null line in $\mathbb{H}^{n+1}_1$, with the exception of one missing point on $y_0+y_{n+1} = 0$
(recall that $f(U^{n+1})=V^{n+1}$ and $f(L^{n+1}_{-})=W^{n+1}$ 
are separated by $y_0+y_{n+1} = 0$ in $\mathbb{H}^{n+1}_1$).
By gluing the upper end of $l$ to the lower end of $l_{-}$, we extend $f$ to map it to the missing point,
and glue $U^{n+1}$ and $L^{n+1}_{-}$ together
(denote the union by $H^{n+1}$) to map it to $\mathbb{H}^{n+1}_1$.
It can be verified that any point on $y_0+y_{n+1} = 0$ in $\mathbb{H}^{n+1}_1$ can be mapped this way,
and it establishes an equivalence relation among those ends in $\mathcal{R}$ and $\mathcal{R}_{-}$.
Similarly, by gluing the lower end of $l$ to the upper end of $l_{-}$ together, 
which obtains a closed null geodesic (see Figure~\ref{figure_null_R_D}),
we glue $U^{n+1}_{-}$ and $L^{n+1}$ together (denote the union by $H^{n+1}_{-}$)
to map it to $\mathbb{H}^{n+1}_{1,-}$.
Finally, we glue $H^{n+1}$ and $H^{n+1}_{-}$ together,
just like how $\mathbb{H}^{n+1}_1$ and $\mathbb{H}^{n+1}_{1,-}$ are glued along 
$\partial\mathbb{H}^{n+1}_1$ to obtain $\mathbb{DH}^{n+1}_1$.

\begin{figure}[h]
\centering
\resizebox{.4\textwidth}{!}
  {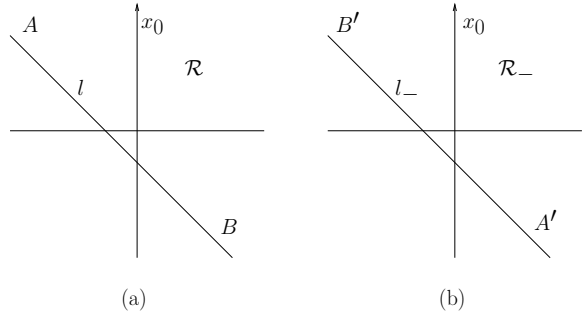}
\caption{A closed null geodesic is formed by connecting a null line
$l$ in $\mathcal{R}$ to $l_{-}$ in $\mathcal{R}_{-}$ through their opposite ends,
by identifying $A$ with $A'$, and $B$ with $B'$ respectively}
\label{figure_null_R_D}
\end{figure}

\begin{definition}
\label{definition_model_R_D}
Denote the resulting space by $\mathcal{R}_D$.
As a new model for $\mathbb{DH}^{n+1}_1$, $\mathcal{R}_D$ contains $\mathcal{R}$ and $\mathcal{R}_{-}$ 
glued along $\partial\mathcal{R}$.
\end{definition}

The mapping $f$ of $\mathcal{R}_D$ to $\mathbb{DH}^{n+1}_1$,
particularly at this moment when we have not analyzed $\partial\mathcal{R}$ in detail yet,
may seem complicated at first sight.
But once the isometries of $\mathcal{R}_D$ to \emph{itself} are established and understood
(see Section~\ref{section_isometries_R_D}),
which for our purpose are a lot easier to use than the isometries of $\mathbb{DH}^{n+1}_1$,
$\partial\mathcal{R}$ can be studied within $\mathcal{R}_D$ itself,
and then we won't need to use the mapping $f$
between $\mathcal{R}_D$ and $\mathbb{DH}^{n+1}_1$ afterwards.

\subsection{A sign function $h(x)$}

To differentiate points in $\mathbb{H}^{n+1}_1$ and $\mathbb{H}^{n+1}_{1,-}$ with the same coordinates,
define 
\begin{equation}
\label{equation_sign_function_AdS}
\ell(y)=1 \quad\text{if}\quad y\in\mathbb{H}^{n+1}_1,
\quad\text{and}\quad 
\ell(y)=-1 \quad\text{if}\quad y\in\mathbb{H}^{n+1}_{1,-},
\end{equation}
and for technical reasons we also define $\ell(y)=0$ for $y\in\partial\mathbb{H}^{n+1}_1$.
Similarly, to differentiate points in $\mathcal{R}$ and $\mathcal{R}_{-}$ with the same coordinates, 
we introduce a sign function $h(x)$ where 
\begin{equation}
\label{equation_sign_function_R_D}
h(x)=1 \quad\text{if}\quad x\in\mathcal{R}, 
\quad\text{and}\quad 
h(x)=-1 \quad\text{if}\quad x\in\mathcal{R}_{-},
\end{equation}
and for technical reasons we also define $h(x)=0$ for $x\in\partial\mathcal{R}$.

\begin{remark}
In general, for a formula involving $\ell(y)$ and $h(x)$, we are mainly interested in the points
in \emph{generic} position, namely $\ell(y)\ne 0$ and $h(x)\ne 0$,
and are not too concerned with those points with  $\ell(y)=0$ or $h(x)=0$.
Also, unlike $\ell(y)$, we remark that $h(x)$ is a model-specific function for $\mathcal{R}_D$ 
and is not invariant under isometry.
\end{remark}

With $h(x)$ encoding the information of whether $x$ is in $\mathcal{R}$ or $\mathcal{R}_{-}$ or $\partial\mathcal{R}$,
the isometry $f$ between $\mathcal{R}_D$ and $\mathbb{DH}^{n+1}_1$ 
can be expressed as
\begin{equation}
\label{equation_R_D_AdS_map_algebraic}
f: (x,h(x))\to (y,\ell(y)),
\end{equation}
where $y$ is defined in (\ref{equation_isometric_embedding}).
By the argument above about where $U^{n+1}$, $L^{n+1}$, $U^{n+1}_{-}$ and $L^{n+1}_{-}$ are mapped to by $f$
(see Figure~\ref{figure_mapping_AdS_R_D}), one immediately sees that
\begin{equation}
\label{equation_l_from_R_D_to_hyperboloid}
\ell(y)=\sgn(x_0) h(x)
\end{equation}
for both $x\in\mathcal{R}$ and $x\in\mathcal{R}_{-}$.

We reserve the notations of $\mathcal{R}$, $\mathcal{R}_{-}$, $\mathcal{R}_D$,
as well as $h(x)$ and $\ell(y)$, for the rest of this paper.
We next show that $h(x)$ of $\mathcal{R}_D$ plays an important role in the isometries of $\mathcal{R}_D$.

\subsection{Isometries of $\mathcal{R}_D$}
\label{section_isometries_R_D}

Recall that an isometry of $\mathbb{DH}^{n+1}_1$ by definition
is uniquely determined by an isometry of $\mathbb{H}^{n+1}_1$.
Because an isometry of $\mathbb{H}^{n+1}_1$ also preserves the antipodal points inside $\mathbb{H}^{n+1}_1$ itself,
so it is further uniquely determined by the isometry restricted to any half-space of $\mathbb{H}^{n+1}_1$.

Assume $g$ is an isometry of $\mathcal{R}_D$ to itself in the following.
By the argument above, $g$ is uniquely determined by its restriction to 
the upper half-space $U^{n+1}$ (see (\ref{equation_upper_lower_half_space})).
It is well known that the (possibly \emph{local}) isometry of $U^{n+1}$
can be expressed as a finite composition of isometries of Minkowski space, similarities, and inversions,
but we need to clarify the notions and make them precise in the context of $\mathcal{R}_D$.
Some of the interpretation below may be new in the literature
and is important for understanding the isometries of $\mathcal{R}_D$.

\begin{remark}
\label{remark_isometry_R_D}
As an isometry of $\mathbb{DH}^{n+1}_1$ does not map any point 
in $\mathbb{H}^{n+1}_1$ into $\mathbb{H}^{n+1}_{1,-}$,
so an isometry $g$ of $\mathcal{R}_D$ does not map any point in $U^{n+1}$ into $L^{n+1}$ or $U^{n+1}_{-}$
(see (\ref{equation_upper_lower_half_space}) and (\ref{equation_upper_lower_half_space_minus})),
but it is ok to map into $L^{n+1}_{-}$.
And similarly for other regions.
\end{remark}

Denote by $s_{\lambda}$ a \emph{similarity} of $\mathcal{R}_D$ that 
\begin{equation}
\label{equation_similarity}
s_{\lambda}: (x,h(x))\to (\lambda x,h(x)).
\end{equation}
By the argument above, note that $s_{\lambda}$ is an isometry of $\mathcal{R}_D$ only for $\lambda>0$.
On the other hand, for a different similarity $s_{\lambda,-}: (x,h(x))\to (\lambda x,-h(x))$,
it is an isometry of $\mathcal{R}_D$ only for $\lambda<0$.

Now consider the isometry of $\mathcal{R}_D$ such that it is also an \emph{inversion} $x\to x/x^2$,
one of the most subtle topics of this paper, 
where we provide a different \emph{definition} and interpretation from the one that the reader may be familiar with.
Denote the inversion by $j$, as it is an isometry of $\mathcal{R}_D$, by Remark~\ref{remark_isometry_R_D}, 
we have
\begin{equation}
\label{equation_inversion_sign}
h(j(x))=\sgn(x^2) h(x)
\end{equation}
for both $x\in\mathcal{R}$ and $x\in\mathcal{R}_{-}$.
Then we have the following definition.

\begin{definition}
\label{definition_Minkowski_inversion}
Let $j$ be an isometry of $\mathcal{R}_D$ and be an \emph{inversion} with $j: x\to x/x^2$,
then $j$ can be expressed as
\begin{equation*}
j: (x,h(x))\to (x/x^2,\sgn(x^2) h(x)).
\end{equation*}
\end{definition}

Hence $j$ maps $\{x\in\mathcal{R}: x^2>0\}$ to itself
but maps $\{x\in\mathcal{R}: x^2<0\}$ to $\{x\in\mathcal{R}_{-}: x^2<0\}$ instead,
and maps the light cone $\{x\in\mathcal{R}: x^2=0\}$ into $\partial\mathcal{R}$.
Similarly $j$ maps $\{x\in\mathcal{R}_{-}: x^2>0\}$ to itself
but maps $\{x\in\mathcal{R}_{-}: x^2<0\}$ to $\{x\in\mathcal{R}: x^2<0\}$ instead,
and maps the light cone $\{x\in\mathcal{R}_{-}: x^2=0\}$ into $\partial\mathcal{R}$.

\begin{remark}
\label{remark_Minkowski_inversion}
For any $x$ in a Minkowski space $\mathbb{R}^{n,1}$ with $x^2\ne 0$,
traditionally the inversion is understood as always mapping $x$ to $x/x^2$ in the \emph{same} $\mathbb{R}^{n,1}$,
no matter $x^2$ is positive or negative.
Under this traditional understanding, the light cone centered at the origin
(including the origin and the ends of the light cone, but with the two ends of any null line treated as the \emph{same} point) 
is mapped by the inversion to a ``conformal infinity'' of $\mathbb{R}^{n,1}$,
forming a \emph{conformal compactification} of $\mathbb{R}^{n,1}$ with topology 
$(\mathbb{S}^n\times\mathbb{S}^1)/\{\pm 1\}$ 
(e.g., see Jadczyk~\cite{Jadczyk:myths}).
Our construction of $\mathcal{R}_D$ is a (conformal) double cover of this conformal compactification
and has topology $\mathbb{S}^n\times\mathbb{S}^1$,
and in $\mathcal{R}$ the two ends of any null line are treated as two different points.
To our knowledge, Definition~\ref{definition_Minkowski_inversion} provides a \emph{new definition} 
of the inversion in Minkowski space, different from the traditional definition.
\end{remark}

Similarly, consider the isometry of $\mathcal{R}_D$ with $x\to -x/x^2$,
and denote it by $j_{-}$. So again by Remark~\ref{remark_isometry_R_D}, we have 
\[h(j_{-}(x))=-\sgn(x^2) h(x)
\]
for both $x\in\mathcal{R}$ and $x\in\mathcal{R}_{-}$.
Then we have the following definition.

\begin{definition}
\label{definition_Minkowski_inversion_negative}
Let $j_{-}$ be an isometry of $\mathcal{R}_D$ with $j_{-}: x\to -x/x^2$,
then $j_{-}$ can be expresses as
\begin{equation*}
j_{-}: (x,h(x))\to (-x/x^2,-\sgn(x^2) h(x)).
\end{equation*}
We also refer to $j_{-}$ as the \emph{inversion with respect to $-1$}.
\end{definition}

Hence $j_{-}$ maps $\{x\in\mathcal{R}: x^2>0\}$ to $\{x\in\mathcal{R}_{-}: x^2>0\}$
and maps $\{x\in\mathcal{R}: x^2<0\}$ to itself, and similarly for $x\in\mathcal{R}_{-}$.

\subsection{Half-spaces in $\mathcal{R}_D$}

Induced by the isometry $f$ (\ref{equation_R_D_AdS_map_algebraic})
between $\mathcal{R}_D$ and $\mathbb{DH}^{n+1}_1$,
a half-space in $\mathcal{R}_D$ is the preimage of a half-space in $\mathbb{DH}^{n+1}_1$ under $f$.
Recall that a half-space in $\mathbb{DH}^{n+1}_1$ is obtained by gluing a half-space in $\mathbb{H}^{n+1}_1$ 
and its antipodal image in $\mathbb{H}^{n+1}_{1,-}$ along $\partial\mathbb{H}^{n+1}_1$,
which can be expressed as
\begin{equation}
\label{equation_half_space_hyperboloid_AdS}
\bigl\{y \in \mathbb{DH}^{n+1}_1: \ell(y)y\cdot e \le 0 \bigr\},
\end{equation}
where $e=(e_0,\dots,e_{n+1})$ is a non-zero vector in $\mathbb{R}^{n,2}$ (but need not be a unit vector),
$y\cdot e$ is the same bilinear product (\ref{equation_bilinear_product}) 
for both $y\in\mathbb{H}^{n+1}_1$ and $y\in\mathbb{H}^{n+1}_{1,-}$,
and $\ell(y)$ is defined in (\ref{equation_sign_function_AdS}).
When $e^2$ is positive or negative or zero, the metric of the face of the half-space
is Lorentzian or Riemannian or degenerate respectively.
For a negative $e^2$, the face of the half-space  contains a pair of $\mathbb{DH}^n$.

The half-space in $\mathcal{R}_D$, as the preimage of (\ref{equation_half_space_hyperboloid_AdS}) 
under $f$ by (\ref{equation_isometric_embedding}), and replacing $\ell(y)$ with $\sgn(x_0) h(x)$
(see (\ref{equation_l_from_R_D_to_hyperboloid})), satisfies
\[\frac{\sgn(x_0)h(x)}{x_0}\left(e_0\frac{1-x^2}{2}-e_1x_1-\cdots - e_{n-1}x_{n-1} + e_nx_n
- e_{n+1}\frac{1+x^2}{2}\right) \le 0.
\]
As the factor $\frac{\sgn(x_0)}{x_0}$ is always positive for $x_0\ne 0$, so it can be dropped.
Then the half-space in $\mathcal{R}_D$ can be written as
\begin{equation}
\label{equation_half_space_R_D}
\{x\in \mathcal{R}_D: h(x) (ax^2+b\cdot x +c) \le 0\},
\end{equation}
where $a=-(e_0+e_{n+1})/2$, $b=(0,-e_1,\dots,-e_n)$, $c=(e_0-e_{n+1})/2$,
and $b\cdot x$ is the bilinear product on $\mathbb{R}^{n,1}$ (\ref{equation_bilinear_product_Minkowski})
for both $x\in\mathcal{R}$ and $x\in\mathcal{R}_{-}$.
We remark that $a$ is allowed to be 0, as long as $a$, $b$, $c$ are not all 0.

\begin{definition}
The \emph{discriminant} of $ax^2+b\cdot x+c$ is defined by $b^2-4ac$.
\end{definition}

Using the formulas above, the discriminant of $ax^2+b\cdot x+c$ 
is also $e_0^2+\cdots +e_{n-1}^2- e_n^2-e_{n+1}^2$,
the same as the bilinear product $e^2$ on $\mathbb{R}^{n,2}$ (\ref{equation_bilinear_product}).
As the metric of the face is determined by $e^2$,
so depending on if $b^2-4ac$ is positive or negative or zero,
the metric of the face is Lorentzian or Riemannian or degenerate respectively.
If $a\ne 0$, when restricted to $\mathcal{R}$ or $\mathcal{R}_{-}$,
the face is either a hyperboloid or a light cone centered on $x_0=0$ 
(but not necessarily at the origin, see Figure~\ref{figure_faces_R_D}).
If $a=0$ and $b\ne 0$, the face is a vertical plane to $x_0=0$.

\begin{figure}[h]
\centering
  \includegraphics[width=0.6\textwidth]{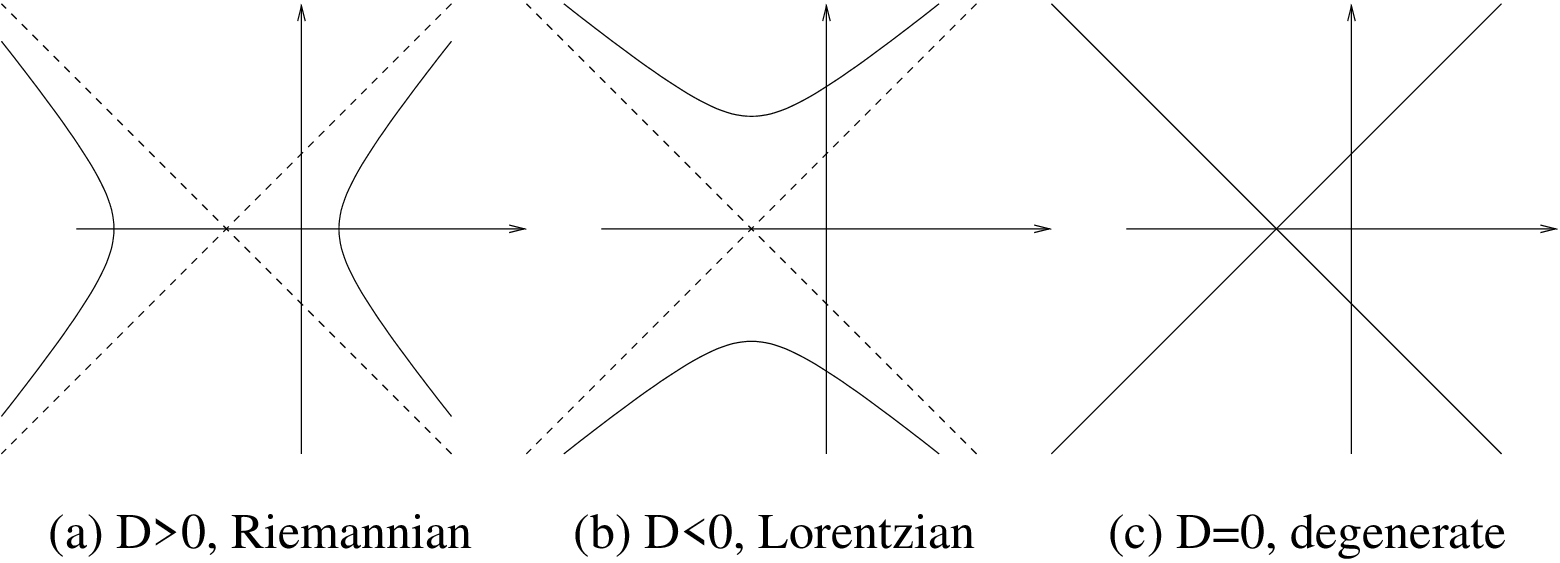}
\caption{The face $ax^2+b\cdot x+c=0$ (when $a\ne 0$) with discriminant $D=b^2-4ac$}
\label{figure_faces_R_D}
\end{figure}

\begin{remark}
\label{remark_discriminant}
If $a=0$ and $b=0$ with $c<0$ (resp. $c>0$),
it corresponds to the half-space $\mathcal{R}$ (resp. $\mathcal{R}_{-}$).
The discriminant is $b^2-4ac=0^2-4\cdot 0\cdot c=0$, a \emph{zero}.
So the metric of the face $\partial\mathcal{R}$ is degenerate,
and therefore $\mathcal{R}$ (or $\mathcal{R}_{-}$) is not a \emph{good} half-space of $\mathcal{R}_D$.
\end{remark}

Notice that  the $x_0$-coordinate of $b$ in (\ref{equation_half_space_R_D}) is 0, 
so a half-space is symmetric with respect to $x_0=0$ in both $\mathcal{R}$ and $\mathcal{R}_{-}$.
But if a point $x\in\mathcal{R}$ is in the interior (resp. exterior) of the half-space, 
because of the $h(x)$ in (\ref{equation_half_space_R_D}),
then the point in $\mathcal{R}_{-}$ with the same coordinate is in the exterior (resp. interior) of the half-space.

\begin{proposition}
\label{proposition_half_space_inversion}
The inversion of the half-space $\{x\in \mathcal{R}_D: h(x) (ax^2+b\cdot x +c) \le 0\}$ by $j$ is 
$\{x\in \mathcal{R}_D: h(x) (cx^2+b\cdot x +a) \le 0\}$,
and by $j_{-}$ is $\{x\in \mathcal{R}_D: h(x) (-cx^2+b\cdot x -a) \le 0\}$,
which are also half-spaces.
\end{proposition}

\begin{proof}
The two cases are very similar, so we only prove the case of $j$.
Denote $\{x\in \mathcal{R}_D: h(x) (ax^2+b\cdot x +c) \le 0\}$ by $H$,
multiply $\sgn(x^2)/x^2$ on both sides we obtain
\[\sgn(x^2)h(x)(ax^2+b\cdot x +c)/x^2 \le 0.
\]
Let $y=j(x)=x/x^2$, then $(ax^2+b\cdot x+c)/x^2=a+b\cdot y+cy^2$,
and by (\ref{equation_inversion_sign}) we have $\sgn(x^2)h(x)=h(j(x))=h(y)$.
So the inversion of $H$ by $j$ is
\[\{y\in \mathcal{R}_D: h(y) (cy^2+b\cdot y +a) \le 0\},
\]
which is also a half-space. Replacing $y$ with $x$, then we finish the proof.
\end{proof}

\begin{remark}
\label{remark_half_space_inversion}
In a Minkowski space $\mathbb{R}^{n,1}$, under the traditional understanding 
the inversion always maps any $x$ with $x^2\ne 0$ to $x/x^2$ in the same $\mathbb{R}^{n,1}$.
Because $x^2$ can be both positive and negative, the traditional inversion of
$\{x\in \mathbb{R}^{n,1}: ax^2+b\cdot x +c \le 0\}$ is \emph{not}
$\{x\in \mathbb{R}^{n,1}: cx^2+b\cdot x +a \le 0\}$,
even after we ignore some points in a lower dimensional region.
But with our new definition of the inversion $j$ and the use of $h(x)$,
the above inversion formula of a half-space in $\mathcal{R}_D$ is very clean
(Proposition~\ref{proposition_half_space_inversion}).
This may even suggest that the inversion $j$ is more ``natural'' than the traditional inversion in Minkowski space.
\end{remark}

\section{A definition of $V_{n+1}(P)$ on $\mathbb{DH}^{n+1}_1$}
\label{section_polytope_volume_AdS}

For a subset $U$ of $\mathbb{DH}^{n+1}_1$, denote the part in $\mathbb{H}^{n+1}_1$ by $U_{+}$,
and the part in $\mathbb{H}^{n+1}_{1,-}$ by $U_{-}$.
In the following we always assume both $U_{+}$ and $U_{-}$ are measurable sets.

\begin{definition}
\label{definition_finite_standard_volume}
If both the volumes $V_{n+1}(U_{+})$ and $V_{n+1}(U_{-})$ are finite,
then we say that $U$ has a finite \emph{standard} volume $V_{n+1}(U)$,
where $V_{n+1}(U)=V_{n+1}(U_{+})+V_{n+1}(U_{-})$.
Let $\mathcal{U}_0$ be the collection of all those subsets $U$ of $\mathbb{DH}^{n+1}_1$
with finite standard volumes.
\end{definition}

The word \emph{standard} is used in contrast to the new volume 
that we are about to introduce to polytopes in $\mathbb{DH}^{n+1}_1$.
For a subset $U$ of $\mathcal{R}$ (see (\ref{equation_Lorentz_metric})), if $U\in \mathcal{U}_0$,
the volume of $U$ can be computed by integrating the volume element 
$dx_0\cdots dx_n/x_0^{n+1}$ of $\mathcal{R}$ over $U$
\begin{equation*}
V_{n+1}(U)=\int_{U\subset \mathcal{R}} \frac{dx_0\cdots dx_n}{x_0^{n+1}},
\end{equation*}
and $V_{n+1}(U)$ is invariant under isometry.

However, for a region $U$ that sits across $x_0=0$ in $\mathcal{R}$ but $U\not\in \mathcal{U}_0$,
the integral is not well defined at $x_0=0$.
To fix this issue and extend the definition of \emph{volume} to more sets that sit across 
$x_0=0$ in $\mathcal{R}$ but \emph{cannot} be computed by the integral above,
we perturb the volume element to obtain a family of complex measures on $\mathcal{R}$,
and define a volume as the limit of the integral of the complex measures, whenever the limit exists.
This is consistent with the methodology used in Zhang~\cite{Zhang:double_hyperbolic}
to extend the volume on the hyperbolic space to the double hyperbolic space.
Cho and Kim~\cite{ChoKim} employed a slightly different methodology to extend the volume 
on the hyperbolic space to the de Sitter space, using the Klein model under a projective extension.
To clear up some confusion, we shall note that while in Lorentzian geometry sometimes 
a timelike interval may be treated as having an \emph{imaginary} length, 
this imaginariness has nothing to do with the complex perturbation.

We first introduce a complex valued measure $\mu$ on $\mathcal{R}$, 
which will eventually lead to an extended definition of a volume $V_{n+1}(U)$ on $\mathbb{DH}^{n+1}_1$.

\subsection{Definition of $\mu(P)$ on $\mathcal{R}$}

For any $\epsilon\ne 0$, on $\mathcal{R}$ we define a complex valued ``Lorentzian metric''
\begin{equation*}
ds_{\epsilon}=(dx_0^2+\cdots +dx_{n-1}^2-dx_n^2)^{1/2}/(x_0-\epsilon i).
\end{equation*}
The associated inner product (to $ds_{\epsilon}$) on the tangent space at a point in $\mathcal{R}$ 
is $1/(x_0-\epsilon i)^2$ times the standard Minkowski inner product.
We shall note that while the metric is complex valued, there is no complex geometry involved.
The metric $ds_{\epsilon}$ is conformally equivalent to the metric of Minkowski space $\mathbb{R}^{n,1}$ 
by a \emph{conformal factor} of $1/(x_0-\epsilon i)$, and the associated volume element of $ds_{\epsilon}$ is 
$\frac{dx_0\cdots dx_n}{(x_0-\epsilon i)^{n+1}}$.

In this paper we treat the (singular) metric $(dx_0^2+\cdots +dx_{n-1}^2-dx_n^2)^{1/2}/x_0$
on $\mathcal{R}$ as a limit of $ds_{\epsilon}$ as $\epsilon\to 0^{+}$,
and introduce a volume on $\mathcal{R}$ as a limit of the complex valued volume induced by $ds_{\epsilon}$.
We caution that as $\epsilon\to 0^{-}$, the limit may be different.
Let $U$ be a subset of $\mathcal{R}$.
If we say $U$ is in a \emph{finite} region in $\mathcal{R}$, which is a model-specific notion for $\mathcal{R}_D$
and not to be confused with the notion of ``in a \emph{bounded} region in $\mathbb{H}^{n+1}_1$ 
or $\mathbb{H}^{n+1}_{1,-}$'',
it is in the sense that when $\mathcal{R}$ is treated as $\mathbb{R}^{n,1}$,
and it is ok for $U$ to contain points in $x_0=0$.
For a subset $U$ of $\mathcal{R}$, define
\begin{equation}
\label{equation_volume_P_half_space_AdS_epsilon}
\mu_{\epsilon}(U):= \int_{U\subset \mathcal{R}}\frac{dx_0\cdots dx_n}{(x_0-\epsilon i)^{n+1}},
\quad\mu(U):=\lim_{\epsilon\to 0^{+}} \mu_{\epsilon}(U),
\end{equation}
whenever the integral exists.
We say $U$ is $\mu$-\emph{measurable} if $\mu(U)$ exists (finite),
and denote by $\mathcal{V}$ the collection of all $\mu$-measurable sets of $\mathcal{R}$.
We caution that this notion of $\mu$-measurable is model-specific for $\mathcal{R}_D$.

For a subset $U$ of $\mathcal{R}$, if $U\in \mathcal{U}_0$,
as $|\frac{1}{(x_0-\epsilon i)^{n+1}}| \le |\frac{1}{x_0^{n+1}}|$,
then the Lebesgue dominated convergence theorem applies to $\mu_{\epsilon}(U)$,
thus $\mu(U)$ exists and $\mu(U)=V_{n+1}(U)$,
and therefore $U$ is $\mu$-measurable.

Our goal is to extend the definition of $V_{n+1}(U)$ to more sets beyond $\mathcal{U}_0$.
To do so, we start with the $\mu$-measurable sets $U$ of $\mathcal{R}$,
and treat $\mu(U)$ as a potential candidate to extend the definition of $V_{n+1}(U)$ to more sets,
but there are some issues that need to be addressed.
First, for any potential extension of $V_{n+1}(U)$ to be meaningful, it should be invariant under isometry,
but we do not know if $\mu(U)$ is invariant under isometry for all $\mu$-measurable sets $U$ in $\mathcal{V}$;
and even if so, $\mathcal{V}$ is not an algebra in the sense that
there are $\mu$-measurable sets $U$ and $U'$ such that $U\cap U'$ is not $\mu$-measurable.

However, if we restrict the extension of $V_{n+1}(U)$ only to the sets generated by \emph{good half-spaces}
in $\mathbb{DH}^{n+1}_1$, we will show that all the issues above are resolved. 
A good half-space is a half-space whose face has non-degenerate metric,
and let $\mathcal{H}$ (resp. $\mathcal{H}_0$) be the algebra over $\mathbb{DH}^{n+1}_1$ 
generated by half-spaces (resp. good half-spaces) in $\mathbb{DH}^{n+1}_1$
(see Definition~\ref{definition_algebra_half_space}),

\begin{theorem}
\label{theorem_mu_volume_invariant_AdS}
For $n\ge 1$, let $P\in\mathcal{H}_0$ in $\mathbb{DH}^{n+1}_1$ be in a finite region in $\mathcal{R}$,
then $\mu(P)$ exists and is invariant under isometries of $\mathcal{R}_D$
(for isometries $g$ with $g(P)$ also in a finite region in $\mathcal{R}$).
\end{theorem}

\begin{remark}
\label{remark_union_polytopes_AdS}
As $P\in\mathcal{H}_0$ is generated by good half-spaces, then by a property in set theory, 
$P$ can be cut into a disjoint union of good polytopes,
where each one is the intersection of some good half-spaces or their complements (which are also good half-spaces).
So to prove Theorem~\ref{theorem_mu_volume_invariant_AdS},
we can reduce it to just considering the case that 
$P$ is a good polytope in a finite region in $\mathcal{R}$.
\end{remark}

The proof constitutes an essential part of the paper
that will run through Section~\ref{section_proof_mu_volume_invariant_AdS}.

\subsection{Proof of Theorem~\ref{theorem_polytope_volume_finite_invariant} 
(assuming Theorem~\ref{theorem_mu_volume_invariant_AdS})}
\label{section_proof_volume_invariant}

Now we introduce $V_{n+1}(P)$ as a measure on the algebra $\mathcal{H}_0$.

\begin{definition}
\label{definition_volume_polytope_AdS}
For $n\ge 1$, let $P\in\mathcal{H}_0$ in $\mathbb{DH}^{n+1}_1$.
We first cut $P$ into a disjoint union of $P_i$'s with $P_i\in\mathcal{H}_0$,
such that for each $i$ there is an isometry $g_i$ that maps $P_i$ into a finite region in $\mathcal{R}$
(this is always doable for all $P$).
Then define $V_{n+1}(P)$ by $\sum \mu(g_i(P_i))$.
\end{definition}

Assuming Theorem~\ref{theorem_mu_volume_invariant_AdS},
we prove Theorem~\ref{theorem_polytope_volume_finite_invariant} below.

\begingroup
\def\thetheorem{\ref{theorem_polytope_volume_finite_invariant}}
\begin{theorem}
Let $P\in\mathcal{H}_0$ in $\mathbb{DH}^{n+1}_1$,
then $V_{n+1}(P)$ is well defined and invariant under isometry.
\end{theorem}
\addtocounter{theorem}{-1}
\endgroup

\begin{proof}
First fix a cut of $P$ with $P_i$'s, and let $g_i$ be an isometry such that $g_i(P_i)$
is in a finite region in $\mathcal{R}$.
Now for a different cut of $P$ with $P'_j$'s and isometries $g'_j$ 
such that $g'_j(P'_j)$ is in a finite region in $\mathcal{R}$, we have 
\[\sum_i \mu(g_i(P_i))=\sum_i \mu(\bigcup_j g_i(P_i\cap P'_j))
=\sum_{i,j} \mu(g_i(P_i\cap P'_j)).
\]
By Theorem~\ref{theorem_mu_volume_invariant_AdS}, $\mu(g_i(P_i\cap P'_j))=\mu(g'_j(P_i\cap P'_j))$, so
\[\sum_i \mu(g_i(P_i))=\sum_{i,j} \mu(g'_j(P_i\cap P'_j))
=\sum_j \mu(\bigcup_i g'_j(P_i\cap P'_j))
=\sum_j \mu(g'_j(P'_j)).
\]
Hence $\sum_i \mu(g_i(P_i))=\sum_j \mu(g'_j(P'_j))$, and therefore $V_{n+1}(P)$ 
is independent of the choice of the cut as well as the isometries $g_i$.
Thus $V_{n+1}(P)$ is well defined and invariant under isometry.
\end{proof}

\subsection{Measure theory on $\mathbb{DH}^{n+1}_1$}

For completeness, we show that $V_{n+1}(U)$ can be further extended to be defined on $\mathcal{H}'_0$,
the algebra over $\mathbb{DH}^{n+1}_1$ generated by $\mathcal{H}_0$ and $\mathcal{U}_0$
(the collection of all the sets in $\mathbb{DH}^{n+1}_1$ with finite standard volume,
see Definition~\ref{definition_finite_standard_volume}).
But once we are done with defining a volume on $\mathcal{H}'_0$,
we will move our focus back to $\mathcal{H}_0$ for the rest of the paper.

If $U\in \mathcal{H}'_0$, denote by $U^c$ the complement of $U$ in $\mathbb{DH}^{n+1}_1$
(not the complement in $\mathbb{H}^{n+1}_1$, even if $U$ is entirely in $\mathbb{H}^{n+1}_1$).

\begin{proposition}
\label{proposition_measurable_algebra_extention}
Assuming Theorem~\ref{theorem_mu_volume_invariant_AdS}, for $n\ge 1$, 
the definition of $V_{n+1}(U)$ can be further extended to $\mathcal{H}'_0$ in $\mathbb{DH}^{n+1}_1$
as a finitely additive measure.
\end{proposition}

\begin{proof}
For any $U\in\mathcal{H}'_0$, 
assume $U$ is generated by good half-spaces $H_1,\dots, H_k$ in $\mathbb{DH}^{n+1}_1$
and $U_1,\dots, U_l\in\mathcal{U}_0$. By a property in set theory,
$U$ is the disjoint union of regions $E$ in the form $(\bigcap_{i=1}^k A_i)\cap (\bigcap_{j=1}^l B_j)$,
where $A_i$ is either $H_i$ or $H_i^c$ (which is also a good half-space in $\mathbb{DH}^{n+1}_1$),
and $B_j$ is either $U_j$ or $U_j^c$ (but $U_j^c$ is not in $\mathcal{U}_0$).
For a region $E$, if at least one $B_j$ is $U_j$, 
then $E$ has finite standard volume, and thus $V_{n+1}(E)$ exists.

Now consider a region $E$ whose $B_j$ are all $U_j^c$.
Denote $\bigcap_{i=1}^k A_i$ by $P$, and $\bigcup_{j=1}^l U_j$ by $U_0$.
Then $P\in\mathcal{H}_0$ and $U_0\in\mathcal{U}_0$, and
\[E=P\cap (\bigcap U_j^c)=P\cap (\bigcup U_j)^c=P\cap U_0^c=P\setminus U_0.
\]
By Theorem~\ref{theorem_polytope_volume_finite_invariant}
(assuming Theorem~\ref{theorem_mu_volume_invariant_AdS}), $V_{n+1}(P)$ is well defined;
and because $P\cap U_0$ has finite standard volume, so $V_{n+1}(P\cap U_0)$ exists.
As $P$ is the disjoint union of $P\cap U_0$ and $E=P\setminus U_0$,
we define $V_{n+1}(E):=V_{n+1}(P)-V_{n+1}(P\cap U_0)$.
Sum up all those regions $E$, we then obtain $V_{n+1}(U)$.
Similar to the proof of Theorem~\ref{theorem_polytope_volume_finite_invariant},
we can verify that $V_{n+1}(U)$ is well defined on $\mathcal{H}'_0$.
\end{proof}

To summarize, to show that $V_{n+1}(U)$ is well defined on $\mathcal{H}'_0$,
all we are left to do is to prove Theorem~\ref{theorem_mu_volume_invariant_AdS}.
The main idea of our proof of Theorem~\ref{theorem_mu_volume_invariant_AdS}
is to intersect a good polytope $P$ in $\mathbb{DH}^{n+1}_1$ with a moving lower dimensional $\mathbb{DH}^n$
and then integrate the volumes of its moving $(n-1)$-dimensional faces.
This reduces a large part of the proof to the (double) hyperbolic case, 
which was already established in \cite{Zhang:double_hyperbolic}.
We provide some necessary background next.

\section{A definition of $V_n(P)$ on $\mathbb{DH}^n$}
\label{section_polytope_volume_DH}

For the construction of the double hyperbolic space $\mathbb{DH}^n$ \cite{Zhang:double_hyperbolic},
one of the main results was that for regions $P$ generated by half-spaces in $\mathbb{DH}^n$,
a \emph{volume} $V_n(P)$ was introduced on $P$, which is invariant under isometry
and also compatible with the volume elements of both $\mathbb{H}^n$ and $\mathbb{H}^n_{-}$
(see Section~\ref{section_background}).
For some basics of the models of the hyperbolic space, see Cannon \emph{et al.}~\cite{Cannon:hyperbolic}.
By convention, we use the same model names of $\mathbb{H}^n$ to describe $\mathbb{DH}^n$,
but the models for $\mathbb{DH}^n$ are not restricted only to the regions as the names may suggest,
e.g., the \emph{hemisphere model} for $\mathbb{DH}^n$ in not restricted to the region of the ``hemisphere''
but uses the full-sphere instead.

Here we use the hemisphere model to introduce some basic notions of $\mathbb{DH}^n$,
but the notions can be easily extended to other models. 
In a Euclidean space $\mathbb{R}^{n+1}$, let $S^n_r$ be a sphere with radius $r$ centered on $x_0=0$.
Using $S^n_r$ as the hemisphere model for $\mathbb{DH}^n$, 
denote the upper hemisphere with  $x_0>0$ by $\mathbb{H}^n$,
and the lower hemisphere with  $x_0<0$ by $\mathbb{H}^n_{-}$,
and they are glued along the boundary $\partial\mathbb{H}^n$ on $x_0=0$ with a natural identification.
The associated metric for both $\mathbb{H}^n$ and $\mathbb{H}^n_{-}$ is
\[ds=(dx_0^2+\cdots +dx_n^2)^{1/2}/x_0.
\]
We note that $\mathbb{H}^n_{-}$ is not isometric to $\mathbb{H}^n$,
and the length element $ds$ on $\mathbb{H}^n_{-}$ is the negative of the $ds$ on $\mathbb{H}^n$.
An \emph{isometry} of $\mathbb{DH}^n$ is an isometry of $\mathbb{H}^n$ 
that also preserves the mirror points in $\mathbb{H}^n_{-}$. 
A \emph{half-space} in $\mathbb{DH}^n$ is obtained by gluing a half-space in $\mathbb{H}^n$
and its mirror image in $\mathbb{H}^n_{-}$ along the boundary $\partial\mathbb{H}^n$.
In $\mathbb{R}^{n+1}$, for any plane vertical to $x_0=0$ and crossing $S^n_r$,
it cuts $\mathbb{DH}^n$ into two half-spaces.
Note that by definition both $\mathbb{H}^n$ and $\mathbb{H}^n_{-}$
are \emph{not} half-spaces in $\mathbb{DH}^n$.
A \emph{polytope} in $\mathbb{DH}^n$ is a finite intersection of half-spaces in $\mathbb{DH}^n$.
By definition a polytope in $\mathbb{DH}^n$ is always symmetric between 
$\mathbb{H}^n$ and $\mathbb{H}^n_{-}$ through a mirror reflection at $x_0=0$,
but a polytope in $\mathbb{H}^n$ by itself is not a polytope in $\mathbb{DH}^n$.

\begin{remark}
In $\mathbb{DH}^n$, any half-space has non-degenerate metric on the face,
so there is no need to specifically introduce the notion of a \emph{good} half-space 
or a \emph{good} polytope, as we did with $\mathbb{DH}^{n+1}_1$
in Definition~\ref{definition_algebra_half_space}.
\end{remark}

In $\mathbb{DH}^n$, for a region $P$ generated by half-spaces,
it was shown in \cite{Zhang:double_hyperbolic} that $V_n(P)$ may be defined by various equivalent methods. 
The primary method is to use the \emph{upper half-space model},
analogous to the way how we define the volume of good polytopes in $\mathbb{DH}^{n+1}_1$.
However, for the convenience of this paper, we choose to define $V_n(P)$ using the hemisphere model,
which will be mainly used as a middle step to help compute the volume on $\mathbb{DH}^{n+1}_1$.
In $S^n_r$, for a region $P$ generated by half-spaces of $\mathbb{DH}^n$, define
\begin{equation}
\label{equation_volume_P_hemisphere_r_epsilon_variant}
\mu'_{h,\epsilon}(P)= \int_{P\subset S^n_r}\pm\frac{rdx_1\cdots dx_n}{(x_0-\epsilon i)^{n+1}},
\quad \mu'_h(P)=\lim_{\epsilon\to 0^{+}} \mu'_{h,\epsilon}(P),
\end{equation}
with the plus sign for $x_0>0$ and the minus sign for $x_0<0$.
When $P$ is a polytope in $\mathbb{DH}^n$, we have the following important property.

\begin{remark}
A similar notion $\mu_h(P)$ was also introduced in the hemisphere model, 
using a volume element $\pm\frac{rdx_1\cdots dx_n}{x_0(x_0-\epsilon i)^n}$
instead of the volume element $\pm\frac{rdx_1\cdots dx_n}{(x_0-\epsilon i)^{n+1}}$ above
in (\ref{equation_volume_P_hemisphere_r_epsilon_variant}).
It can also be used to define $V_n(P)$, but we won't use $\mu_h(P)$ in this paper.
\end{remark}

\begin{theorem}[\emph{\cite[Theorem~4.6]{Zhang:double_hyperbolic}}]
\label{theorem_volume_invariant}
For $n\ge 0$, let $P$ be a polytope in $\mathbb{DH}^n$.
\begin{enumerate}[(1).]
\item Then in the hemisphere model, $\mu'_h(P)$ exists and is invariant under isometry.
\item \emph{(Uniform boundedness for a fixed $m$)}
For a fixed $m$, if $P$ is the intersection of at most $m$ half-spaces, 
then $\mu'_{h,\epsilon}(P)$ is uniformly bounded for all $P$ and $\epsilon>0$,
as well as $r>0$.
\end{enumerate}
\end{theorem}

\begin{definition}
\label{definition_volume_P}
Let $P$ be a region in $\mathbb{DH}^n$ generated by half-spaces, 
then in the hemisphere model define $V_n(P)$ by $\mu'_h(P)$.
For $n=0$, $V_0(P)$ is the number of points in $P$.
In other models, we first map $P$ to the hemisphere model by an isometry $g$,
then define $V_n(P)$ by $\mu'_h(g(P))$.
\end{definition}

Particularly, we have the following property of $V_n(P)$ with an explicit bound.

\begin{theorem}[\emph{\cite[Theorem~1.3]{Zhang:double_hyperbolic}}, Uniform boundedness of $V_n(P)$ for a fixed $m$]
\label{theorem_volume_bounded}
For a fixed $m$, let $P$ be a polytope in $\mathbb{DH}^n$
and be the intersection of at most $m$ half-spaces in $\mathbb{DH}^n$, 
then $V_n(P)$ is uniformly bounded for all $P$ with
$|V_n(P)|\le \frac{m!}{2^{m-1}} V_n(\mathbb{S}^n)$.
\end{theorem}

\section{Proof of Theorem~\ref{theorem_mu_volume_invariant_AdS}}
\label{section_proof_mu_volume_invariant_AdS}

We now start proving Theorem~\ref{theorem_mu_volume_invariant_AdS}.
By Remark~\ref{remark_union_polytopes_AdS},
we only need to consider the case that $P$ is a good polytope in a finite region in $\mathcal{R}$ 
(see (\ref{equation_Lorentz_metric})).

\subsection{Existence of $\mu(P)$}

Let $P$ satisfy $t_0\le x_n\le t_1$ in $\mathcal{R}$,
and $P_t$ be the intersection of $P$ and a moving half-space $x_n\le t$
(more precisely, the moving half-space is $\{x\in \mathcal{R}_D: h(x)(x_n-t)\le 0\}$ in $\mathcal{R}_D$,
see (\ref{equation_half_space_R_D})).
By (\ref{equation_volume_P_half_space_AdS_epsilon}) we have
\begin{equation}
\label{equation_volume_P_half_space_AdS_epsilon_copy}
\mu_{\epsilon}(P_t)= \int_{P_t\subset \mathcal{R}}\frac{dx_0\cdots dx_n}{(x_0-\epsilon i)^{n+1}},
\quad\mu(P_t)=\lim_{\epsilon\to 0^{+}} \mu_{\epsilon}(P_t).
\end{equation}
Let $E_t$ be the $n$-face of $P_t$ in the plane $x_n=t$ in $\mathcal{R}$, and
\begin{equation}
\label{equation_volume_epsilon_integration_derivative}
b_{\epsilon}(t):=\int_{E_t}\frac{dx_0\cdots dx_{n-1}}{(x_0-\epsilon i)^{n+1}}
\quad\text{and}\quad
\mu_{\epsilon}(P_t)=\int_{t_0}^{t}b_{\epsilon}(t)dt.
\end{equation}
On the right side of $b_{\epsilon}(t)$,
integrating with respect to $x_0$, we have
\begin{equation}
\label{equation_volume_derivative_epsilon}
b_{\epsilon}(t)
=-\frac{1}{n}\int_{\partial E_t}\pm\frac{dx_1\cdots dx_{n-1}}{(x_0-\epsilon i)^n},
\end{equation}
with a plus sign on a point in $\partial E_t$ if the $x_0$ direction points outward to $E_t$
and a minus sign if the $x_0$ direction points inward.
For any \emph{flat} $(n-1)$-face $F$ of $E_t$, on which the form $dx_1\cdots dx_{n-1}$ is 0,
it does not contribute to (\ref{equation_volume_derivative_epsilon}).
For any other $(n-1)$-face $F$ of $E_t$, it is on an $(n-1)$-dimensional sphere centered on $x_0=0$,
and assume the \emph{Euclidean} radius of $F$ is $r_F$.
By applying (\ref{equation_volume_P_hemisphere_r_epsilon_variant}) 
to (\ref{equation_volume_derivative_epsilon}), we have
\begin{equation}
\label{equation_volume_epsilon}
b_{\epsilon}(t)
=-\frac{1}{n}\sum_{F\subset E_t}\pm\frac{1}{r_F}\mu'_{h,\epsilon}(F),
\end{equation}
with the plus sign for \emph{top} faces and the minus sign for \emph{bottom} faces $F$ of $E_t$,
which we define below.
We remark that the ``$\pm$'' in (\ref{equation_volume_P_hemisphere_r_epsilon_variant}) 
and (\ref{equation_volume_derivative_epsilon}) have different meanings, 
and they offset to get the  ``$\pm$'' in (\ref{equation_volume_epsilon}), whose meaning is also different.

\begin{definition}
\label{definition_face_top_bottom_side}
Fix a half-space $\{x\in \mathcal{R}_D: h(x) (ax^2+b\cdot x +c) \le 0\}$
in $\mathcal{R}_D$ (see (\ref{equation_half_space_R_D})),
and assume $a$ and $b$ are not both 0.
When restricted to $\mathcal{R}$, if its face is a vertical plane to $x_0=0$ (namely $a=0$),
then we call it a \emph{side} face in  $\mathcal{R}$.
Otherwise, for a point on its face in $\mathcal{R}$,
when it moves along the $x_0$-axis away from $x_0=0$,
if the point stays outside (resp. inside) the half-space,
then we call the face in $\mathcal{R}$ a \emph{top} (resp. \emph{bottom}) face in $\mathcal{R}$,
indifferent to the sign of $x_0$ (see Figure~\ref{figure_faces_AdS}).
\end{definition}

\begin{figure}[h]
\centering
  \includegraphics[width=0.25\textwidth]{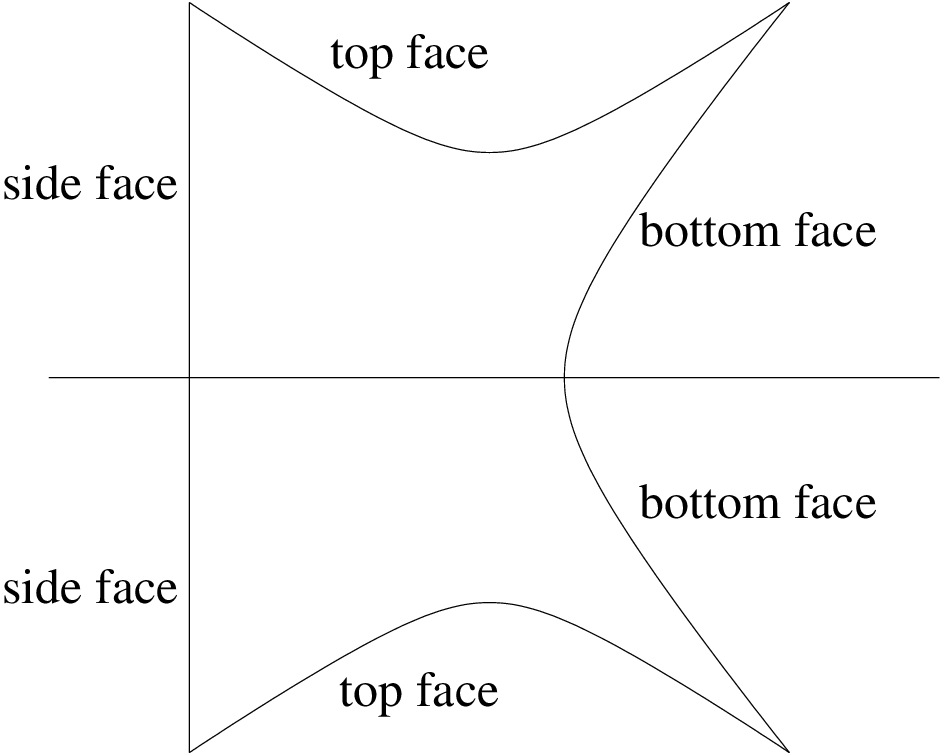}
\caption{The top, bottom, and side faces in $\mathcal{R}$}
\label{figure_faces_AdS}
\end{figure}

It is worth noting that these notions are only meaningful 
when the half-space is restricted to $\mathcal{R}$
(as opposed to considering the whole half-space in $\mathcal{R}_D$),
because when the same half-space is restricted to $\mathcal{R}_{-}$,
the meaning of top and bottom faces is switched.
Because $E_t$ is a side face of $P_t$ in $\mathcal{R}$,
the notions of top, bottom and side faces can also be applied to the $(n-1)$-faces $F$ of $E_t$.

Let $b(t)$ be the pointwise limit of $b_{\epsilon}(t)$ as $\epsilon\to 0^{+}$
\begin{equation}
\label{equation_derivative_special_limit}
b(t):=\lim_{\epsilon\to 0^{+}} b_{\epsilon}(t),
\end{equation}
then by applying Theorem~\ref{theorem_volume_invariant} and Definition~\ref{definition_volume_P}
to (\ref{equation_volume_epsilon}), we have
\[b(t)=-\frac{1}{n}\sum_{F\subset E_t}\pm\frac{1}{r_F}V_{n-1}(F).
\]
Notice that $b(t)$ is continuous for $t$ except at a finite number of points where $E_t$ changes its combinatorial type.
We note that $b(t)$ may also not be bounded at points where $r_F=0$,
but we next show that $b(t)$ is still integrable.

\begin{lemma}
\label{lemma_volume_mu_exist}
For $n\ge 1$, let $P$ in $\mathbb{DH}^{n+1}_1$
be a good polytope in a finite region in $\mathcal{R}$ that satisfies $t_0\le x_n\le t_1$,
and $P_t$ be the intersection of $P$ and $x_n\le t$ in $\mathcal{R}$, with $E_t$ the $n$-face of $P_t$ on $x_n=t$.
\begin{enumerate}[(1).]
\item Let $F$ be any non-side $(n-1)$-face of $E_t$ with radius $r_F$,
then $\frac{1}{r_F}$ is integrable over $t$.
\item Then $b(t)$ is integrable and $\mu(P_t)$ is continuous for $t$ with $\mu(P_t)=\int_{t_0}^t b(t)dt$, and $\mu(P)$ exists.
\end{enumerate}
\end{lemma}

\begin{proof}
Let $F$ be the intersection of $E_t$ and a non-side $n$-face of $P$.
As $P$ is a good polytope, so the $n$-face has non-degenerate metric,
which is either Lorentzian or Riemannian.

If the $n$-face's metric is Lorentzian,
then the $n$-face in $\mathcal{R}$ is on a surface $(x-v)^2=r^2$ with $r>0$,
where $v$ is a vector in the plane $x_0=0$ with $c$ the $x_n$-coordinate of $v$.
Because $x_n$ is a timelike direction, so the component of $r^2$ on the $x_n$ direction is $-(t-c)^2$,
and
\begin{equation}
\label{equation_radius_Lorentzian}
r_F^2=r^2+(t-c)^2,
\end{equation}
Therefore $r_F\ge r$ and $\frac{1}{r_F}\le \frac{1}{r}$.
As $r$ is a constant independent of $t$ and $P$ is in a finite region in $\mathcal{R}$,
so $\frac{1}{r_F}$ is integrable over $t$.

If the $n$-face's metric is Riemannian,
then the $n$-face in $\mathcal{R}$ is on a surface $(x-v)^2=-r^2$ with $r>0$,
where $v$ is a vector in the plane $x_0=0$ with $c$ the $x_n$-coordinate of $v$.
So
\[r_F^2=-r^2+(t-c)^2.
\]
This surface in $\mathcal{R}$ is a two-sheeted hyperboloid, without loss of generality,
by symmetry we only consider the sheet in the positive direction of $x_n$-axis and assume $t\ge c+r$.
Thus $r_F=((t-c)^2-r^2)^{1/2}=(t-c-r)^{1/2}(t-c+r)^{1/2}$.
When $t$ is close to $c+r$, $\frac{1}{r_F}$ is in the order of $O((t-c-r)^{-1/2})$,
whose integration over $t$ is in the order of $O((t-c-r)^{1/2})$,
so $\frac{1}{r_F}$ is integrable over $t$. 

Let $P$ be the intersection of at most $m$ half-spaces in $\mathbb{DH}^{n+1}_1$.
By Theorem~\ref{theorem_volume_invariant},
there is a constant $c_1>0$, depending only on $m$ but not $\epsilon$, 
such that for all non-side $(n-1)$-faces $F$ of $E_t$ (and for all $t$) we have
$|\mu'_{h,\epsilon}(F)|\le c_1$. 
Set $g(t)=\frac{c_1}{n}\sum_{F\subset E_t} \frac{1}{r_F}$, 
then by (\ref{equation_volume_epsilon}) we have $|b_{\epsilon}(t)| \le g(t)$.
As $\frac{1}{r_F}$ is integrable over $t$, so $g(t)$ is integrable.
Therefore Lebesgue dominated convergence theorem applies to $\{b_{\epsilon}(t)\}$ and $b(t)$ is integrable.
Hence by (\ref{equation_volume_P_half_space_AdS_epsilon_copy}),
(\ref{equation_volume_epsilon_integration_derivative}) and (\ref{equation_derivative_special_limit}),
we have
\[\mu(P_t)=\lim_{\epsilon\to 0^{+}} \mu_{\epsilon}(P_t)
=\lim_{\epsilon\to 0^{+}}\int_{t_0}^t b_{\epsilon}(t)dt=\int_{t_0}^t b(t)dt.
\]
So $\mu(P_t)$ is continuous for $t$ and $\mu(P)$ exists. 
\end{proof}

\begin{remark}
\label{remark_half_space_degenerate}
In Lemma~\ref{lemma_volume_mu_exist}, it is important that we require $P\in\mathcal{H}_0$
instead of just $P\in\mathcal{H}$, namely, require the $n$-faces of $P$ to have non-degenerate metrics.
If not, assume $P$ is a polytope in $\mathcal{R}$ bounded by $0\le x_n\le 1$ and $x^2\le 0$
(whose face is $x^2=0$ and has degenerate metric).
Let $F_0$ be the intersection $x^2=0$ and $x_n=t$, then $r_{F_0}=|t|$.
Thus $\frac{1}{r_{F_0}}=\frac{1}{|t|}$, not integrable at $t=0$, and $\mu(P)$ indeed does not exist.
\end{remark}

\begin{remark}
\label{remark_volume_unbounded}
Unlike in $\mathbb{DH}^n$ where a polytope's volume is uniformly bounded 
if the polytope's number of facets is bounded (Theorem~\ref{theorem_volume_bounded}),
we do not have a similar \emph{uniform boundedness} result for $\mathbb{DH}^{n+1}_1$.
Let $P_t$ in $\mathbb{DH}^{n+1}_1$ be a good polytope in a finite region in $\mathcal{R}$
bounded by $x^2\le 1$ and $0\le x_n\le t$.
By Lemma~\ref{lemma_volume_mu_exist}, we have
$\mu(P_t)=-\frac{1}{n}V_{n-1}(\mathbb{DH}^{n-1}) \int_0^t\frac{1}{r_F} dt$, 
where $r_F=(1+t^2)^{1/2}$ (see (\ref{equation_radius_Lorentzian})).
As $t\to +\infty$, $\mu(P_t)$ is not bounded.
\end{remark}

\subsection{Invariance of $\mu(P)$}

To prove the invariance of $\mu(P)$, we first show that $\mu(P)$ 
is invariant under some basic isometries of $\mathcal{R}_D$ (see Section~\ref{section_isometries_R_D}).

\begin{lemma}
\label{lemma_preserve_x0}
For $n\ge 1$, let $g$ be an isometry of $\mathcal{R}_D$.
If $g$ also preserves $\mathcal{R}$ and is an isometry of $\mathbb{R}^{n,1}$,
then for each point in $\mathcal{R}$, $g$ preserves its $x_0$-coordinate.
\end{lemma}

\begin{proof}
In $\mathcal{R}$, for any point $A$ with coordinate $(x_0,x_1\dots,x_n)$,
let $A_0=(0,x_1\dots,x_n)$, then $A-A_0$ is orthogonal to $x_0=0$.
Since $g$ is an isometry of $\mathcal{R}_D$, so $g$ preserves $x_0=0$ and $g(A_0)$ is in $x_0=0$.
As $g$ is also an isometry of $\mathbb{R}^{n,1}$, 
so $g(A)-g(A_0)$ is orthogonal to $x_0=0$ and has the same length as $A-A_0$.
Then $g(A)-g(A_0)$ is either $(x_0,0,\dots,0)$ or $(-x_0,0,\dots,0)$.
As an isometry of $\mathcal{R}_D$ does not map any point in the upper half-space $x_0>0$ 
into the lower half-space $x_0<0$ in $\mathcal{R}$ (by Remark~\ref{remark_isometry_R_D}),
so $g(A)-g(A_0)$ must be $(x_0,0,\dots,0)$, and thus $g(A)$ preserves the $x_0$-coordinate of $A$.
\end{proof}

More generally, if $g$ preserves $\mathcal{R}$ (so $g$ also preserves $\partial\mathcal{R}$),
then we have the following.

\begin{lemma}
\label{lemma_invariance_preserve_R}
For $n\ge 1$, let $P$ in $\mathbb{DH}^{n+1}_1$
be a good polytope in a finite region in $\mathcal{R}$.
If $g$ is an isometry of $\mathcal{R}_D$ that also preserves $\mathcal{R}$,
then $g(P)$ is in a finite region in $\mathcal{R}$ and $\mu(g(P))=\mu(P)$.
\end{lemma}

\begin{proof}
As $g$ preserves $\mathcal{R}$, by a property of Minkowski space, 
$g$ can be written as a combination of  an isometry of Minkowski space and a similarity, with no inversion involved.
So $g(P)$ is in a finite region in $\mathcal{R}$.
In $\mathcal{R}$, let $g(x_0,\dots,x_n)=(y_0,\dots,y_n)$, then by Lemma~\ref{lemma_preserve_x0}, 
there is a constant $\lambda$, such that $y_0=\lambda x_0$.
By Remark~\ref{remark_isometry_R_D}, $\lambda>0$.
As $g$ is an isometry of $\mathcal{R}_D$,
so $g^{\ast}$ maps the volume element $\frac{dy_0\cdots dy_n}{y_0^{n+1}}$ in $\mathcal{R}$ 
into the volume element $\frac{dx_0\cdots dx_n}{x_0^{n+1}}$ (also in $\mathcal{R}$)
with the appropriate orientation of each coordinate system. 
By (\ref{equation_volume_P_half_space_AdS_epsilon}),
\[\mu_{\epsilon}(g(P))
=\int_{g(P)\subset \mathcal{R}}\frac{dy_0\cdots dy_n}{(y_0-\epsilon i)^{n+1}}
=\int_{P\subset \mathcal{R}}\frac{y_0^{n+1}dx_0\cdots dx_n}{(y_0-\epsilon i)^{n+1}x_0^{n+1}}
=\int_{P\subset \mathcal{R}}\frac{dx_0\cdots dx_n}{(x_0-\lambda^{-1}\epsilon i)^{n+1}},
\]
where the last step is because $y_0=\lambda x_0$.
Because $\lambda>0$, so as $\epsilon\to 0^{+}$, 
by applying Lemma~\ref{lemma_volume_mu_exist} to both $g(P)$ and $P$,
we have $\mu(g(P))=\mu(P)$.
\end{proof}

\begin{remark}
\label{remark_invariant_preserve_R}
As a special case, if $g$ is an isometry of $\mathcal{R}_D$ that is also either an isometry of $\mathbb{R}^{n,1}$
or a  similarity $s_{\lambda}$ with $\lambda>0$ (see (\ref{equation_similarity})), then $\mu(g(P))=\mu(P)$.
\end{remark}

Our next goal is to show that $\mu(P)$ is invariant under inversion 
(Lemma~\ref{lemma_mu_invariant_inversion}),
the most important and difficult step to prove the invariance of $\mu(P)$.
For $r>0$, let $P'_r$ be the intersection of $P$ and $x^2\le r^2$ in $\mathcal{R}$,
and $P'_{r,-}$ be the intersection of $P$ and $x^2\le -r^2$ in $\mathcal{R}$.
We first have the following result for $P'_r$, a middle step before proving
Lemma~\ref{lemma_mu_invariant_inversion}.

\begin{lemma}
\label{lemma_polytope_radial_continuous_lorentzian}
For $n\ge 1$, let $P$ in $\mathbb{DH}^{n+1}_1$
be a good polytope in a finite region in $\mathcal{R}$,
then $\mu(P'_r)$ is continuous for $r>0$.
\end{lemma}

\begin{proof}
Assume $P$ satisfies $0\le x_n\le t_1$ in $\mathcal{R}$ (the case of $x_n\le 0$ can be proved similarly).
Let $E_t$ be the intersection of $P$ and $x_n=t$.
For $r>0$, let $E_t(r)$ be the intersection of $P'_r$ and $x_n=t$. Let
\begin{equation}
\label{equation_a_r_t}
a_r(t):=-\frac{1}{n}\sum_{F\subset E_t(r)}\pm\frac{1}{r_F}V_{n-1}(F),
\end{equation}
with the plus sign for top faces and the minus sign for bottom faces $F$ of $E_t(r)$
(see Definition~\ref{definition_face_top_bottom_side}), and $r_F$ is the radius of $F$.
Let $G$ be the $(n-1)$-face of $E_t(r)$ on $x^2=r^2$, and $r_G$ be the radius of $G$.
For a fixed $r_0>0$, denote $G$ at $r=r_0$ by $G_0$, and the radius of $G_0$ by $r_{G_0}$.
Applying Lemma~\ref{lemma_volume_mu_exist} to $a_r(t)$ (with $r$ fixed), then $\mu(P'_r)=\int_{0}^{t_1} a_r(t)dt$.
Our goal is to show that $\mu(P'_r)$ is continuous at $r=r_0$.

Let $P$ be the intersection of at most $m$ half-spaces,
then $P'_r$ is the intersection of at most $m+1$ half-spaces.
Then for any $(n-1)$-face $F$ of $E_t(r)$, by Theorem~\ref{theorem_volume_bounded}, 
there is a constant $c$ depending only on $m$, such that $|V_{n-1}(F)|\le c$. So
\begin{equation}
\label{equation_a_r_t_bound}
|a_r(t)| \le \frac{c}{n}\left(\sum_{F\subset E_t(r)} \frac{1}{r_F}\right)
\le \frac{c}{n}\left(\sum_{F\subset E_t} \frac{1}{r_F} +\frac{1}{r_G}\right).
\end{equation}

Recall that $P'_r$ is the intersection of $P$ and $x^2\le r^2$ in $\mathcal{R}$.
As $x_n$ is a timelike direction, so $r_G^2=r^2+t^2$, therefore $r_G\ge r$.
For a fixed $r_0>0$, we have $r_{G_0}\ge r_0>0$ for the entire region of $t\in [0,t_1]$.
Thus when $r$ is close to $r_0$, we have $r_G\ge \frac{r_{G_0}}{2}$ for \emph{all} $t\in [0,t_1]$, so 
\[|a_r(t)| \le \frac{c}{n}\left(\sum_{F\subset E_t} \frac{1}{r_F} +\frac{2}{r_{G_0}}\right).
\]
Notice that $r_{G_0}$ depends only on $t$ but not $r$, and so does $r_F$ for $F\subset E_t$.
Denote the right side by $g(t)$,
then by Lemma~\ref{lemma_volume_mu_exist}, $g(t)$ is integrable over $t$.
Thus on the entire region of $[0,t_1]$,  as $r\to r_0$,
Lebesgue dominated convergence theorem applies to $\{a_r(t)\}$, and
\[\lim_{r\to r_0}\int_{0}^{t_1} a_r(t)dt = \int_{0}^{t_1} a_{r_0}(t)dt.
\]
As $\mu(P'_r)=\int_{0}^{t_1} a_r(t)dt$, so $\mu(P'_r)$ is continuous at $r=r_0$.
\end{proof}

We have a similar result for $P'_{r,-}$, another middle step before proving
Lemma~\ref{lemma_mu_invariant_inversion}, with a very similar but slightly different proof 
than that of Lemma~\ref{lemma_polytope_radial_continuous_lorentzian}.

\begin{lemma}
\label{lemma_polytope_radial_continuous_riemannian}
For $n\ge 1$, let $P$ in $\mathbb{DH}^{n+1}_1$
be a good polytope in a finite region in $\mathcal{R}$,
then $\mu(P'_{r,-})$ is continuous for $r>0$.
\end{lemma}

\begin{proof}
Assume $P$ satisfies $0\le x_n\le t_1$ in $\mathcal{R}$ (the case of $x_n\le 0$ can be proved similarly).
Let $E_t$ be the intersection of $P$ and $x_n=t$.
For $r>0$, let $E_t(r)$ be the intersection of $P'_{r,-}$ and $x_n=t$. 
To reduce repetition of the proof, with a slight abuse of notation, except that $P'_{r,-}$ replaces $P'_r$,
for convenience we use the same notations as in the proof of Lemma~\ref{lemma_polytope_radial_continuous_lorentzian},
e.g., $G$, $r_G$, $r_0$, $G_0$, $r_{G_0}$, and $a_r(t)$ as in (\ref{equation_a_r_t}).
Applying Lemma~\ref{lemma_volume_mu_exist} to $a_r(t)$ (with $r$ fixed), 
then $\mu(P'_{r,-})=\int_{0}^{t_1} a_r(t)dt$.
For a fixed $r_0>0$, our goal is to show that $\mu(P'_{r,-})$ is continuous at $r=r_0$.
By Theorem~\ref{theorem_volume_bounded}, the same as (\ref{equation_a_r_t_bound}),
there is a constant $c$ depending only on $m$, such that 
\begin{equation}
\label{equation_radius_inverse_estimate}
|a_r(t)| \le \frac{c}{n}\left(\sum_{F\subset E_t(r)} \frac{1}{r_F}\right)
\le \frac{c}{n}\left(\sum_{F\subset E_t} \frac{1}{r_F} +\frac{1}{r_G}\right).
\end{equation}

Now the proof starts to differ from the proof of Lemma~\ref{lemma_polytope_radial_continuous_lorentzian},
mainly because we cannot modify the $\frac{1}{r_G}$ term above to apply 
Lebesgue dominated convergence theorem to $\{a_r(t)\}$ for the \emph{entire} region of $t\in [0,t_1]$.
Recall that $P'_{r,-}$ is the intersection of $P$ and $x^2\le -r^2$ in $\mathcal{R}$.
As $x_n$ is a timelike direction, so $r_G^2=-r^2+t^2$, 
and as we only consider $t\ge 0$, hence
\[r_G=(t^2-r^2)^{1/2}=(t-r)^{1/2}(t+r)^{1/2}.
\]
Thus when $t$ is near $r$, we have $\frac{1}{r_G}=O((t-r)^{-1/2})$ and
\[\int_{r}^{t}\frac{1}{r_G}dt=O((t-r)^{1/2}),
\]
which converges to $0$ as $(t-r)\to 0^{+}$.
For the right side of (\ref{equation_radius_inverse_estimate}), except for $1/r_G$,
the remaining part does not depend on $r$, 
and by Lemma~\ref{lemma_volume_mu_exist} is integrable over $t$.
Hence for a fixed $r_0>0$, for any $\epsilon>0$, there is $\delta>0$, such that for \emph{all} $r>0$, we have
\begin{equation}
\label{equation_integral_estimate_middle}
\int_{r_0-\delta}^{r_0+\delta} |a_r(t)|\,dt < \epsilon.
\end{equation}

For $t\in [0,r_0-\delta]$, when $r>r_0-\delta$,
$\frac{1}{r_G}$ does not contribute to (\ref{equation_radius_inverse_estimate}).
For $t\in [r_0+\delta, t_1]$, notice that $r_{G_0}>0$.
Thus when $r\to r_0$, we have $r_G\ge \frac{r_{G_0}}{2}$ for $t\in [r_0+\delta, t_1]$ 
(see Figure~\ref{figure_polar_AdS}, but we caution that this is not so for \emph{all} $t\in [r_0,t_1]$).
So for the combined region of $[0,r_0-\delta]$ and $[r_0+\delta, t_1]$ (denoted by $U$),
we have
\[|a_r(t)| \le \frac{c}{n}\left(\sum_{F\subset E_t} \frac{1}{r_F} +\frac{2}{r_{G_0}}\right).
\] 
Denote the right side by $g(t)$, 
then by Lemma~\ref{lemma_volume_mu_exist}, $g(t)$ is integrable over $t$ on $U$.
Let $r\to r_0$, Lebesgue dominated convergence theorem applies to $\{a_r(t)\}$ on $U$, and
\begin{equation}
\label{equation_integral_estimate_middle_outside}
\lim_{r\to r_0} \int_{U} a_r(t)\,dt
= \int_{U} a_{r_0}(t)\,dt.
\end{equation}

Let $\epsilon\to 0$, by combining (\ref{equation_integral_estimate_middle})
and (\ref{equation_integral_estimate_middle_outside}), we have 
\[\lim_{r\to r_0} \int_{0}^{t_1} a_r(t)\,dt
= \int_{0}^{t_1} a_{r_0}(t)\,dt.
\]
As $\mu(P'_{r,-})=\int_{0}^{t_1} a_r(t)dt$, so $\mu(P'_{r,-})$ is continuous at $r=r_0$. 
\end{proof}

\begin{figure}[h]
\centering
\resizebox{.25\textwidth}{!}
  {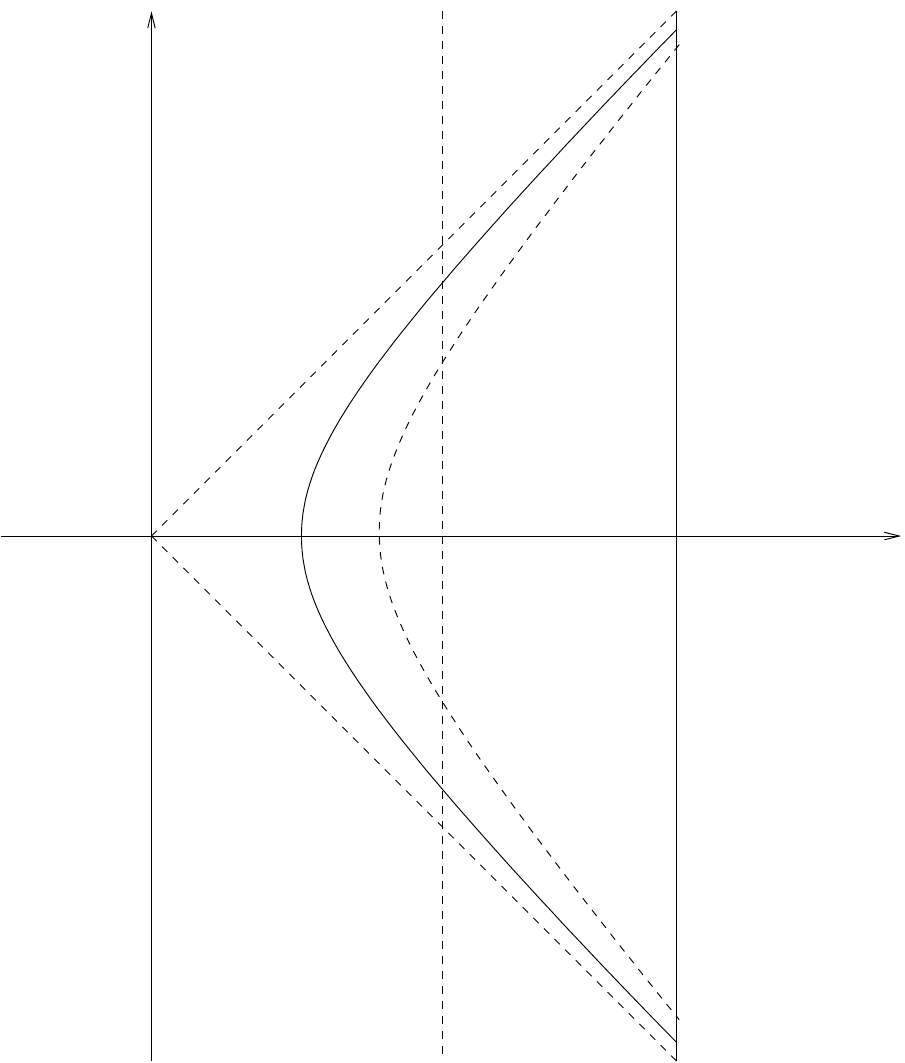}
\caption{The $(n-1)$-face $G_0$ with radius $r_{G_0}$,
on the intersection of $x^2=-r_0^2$ and $x_n=t$ in $\mathcal{R}$}
\label{figure_polar_AdS}
\end{figure}

Now we are ready to show that $\mu(P)$ is invariant under inversion for both $j$ and $j_{-}$
(see Definition~\ref{definition_Minkowski_inversion} and \ref{definition_Minkowski_inversion_negative}).

\begin{lemma}
\label{lemma_mu_invariant_inversion}
For $n\ge 1$, let $P$ and $Q$ in $\mathbb{DH}^{n+1}_1$
be good polytopes in a finite region in $\mathcal{R}$,
and $Q$ is the inversion of $P$ with either $Q=j(P)$ or $Q=j_{-}(P)$, then $\mu(P)=\mu(Q)$.
\end{lemma}

\begin{proof}
We first consider the case of $Q=j(P)$,
which means that both $P$ and $Q$ are inside $\{x\in\mathcal{R}: x^2>0\}$.
For $r>0$, let $P'_r$ be the intersection of $P$ and $x^2\le r^2$ in $\mathcal{R}$,
and $Q'_r$ be the intersection of $Q$ and $x^2\le r^2$ in $\mathcal{R}$.
For $\Delta r>0$, denote $P'_{r+\Delta r}\setminus P'_r$ by $U$.
Except at a finite number of $r$'s where $\mu(P'_r)$ is not differentiable over $r$,
when $\Delta r\to 0$,
\begin{equation}
\label{equation_P_slice_approximation}
\mu(U)=\mu(P'_{r+\Delta r})-\mu(P'_r) 
\approx \frac{d\mu(P'_r)}{dr} \Delta r.
\end{equation}
Now scale $U$ by a factor of $1/r$ at the origin, and denote the resulting region by $U'$.
As $\mu$ is invariant under similarity (see Remark~\ref{remark_invariant_preserve_R}), so $\mu(U')=\mu(U)$.
Denote $Q'_{1/r}\setminus Q'_{1/(r+\Delta r)}$ by $V$.
Similarly, except at a finite number of $r$'s, when $\Delta r\to 0$,
\begin{equation}
\label{equation_Q_slice_approximation}
\mu(V)=\mu(Q'_{1/r})-\mu(Q'_{1/(r+\Delta r)})
\approx -\frac{d\mu(Q'_{1/r})}{dr} \Delta r.
\end{equation}

Scale $V$ by a factor of $r$ at the origin and denote the resulting region by $V'$,
as $\mu$ is invariant under similarity, so $\mu(V')=\mu(V)$.
Because $Q=j(P)$, so $V=j(U)$ and thus $V'=j(U')$.
Hence $U'$ and $V'$ are reflections of each other in the radial direction near the surface $x^2=1$,
therefore $\mu(U')\approx \mu(V')$ as $\Delta r\to 0$, and thus $\mu(U)\approx \mu(V)$.

Compare (\ref{equation_P_slice_approximation}) to (\ref{equation_Q_slice_approximation}),
then
\[\frac{d\mu(P'_r)}{dr} \Delta r
\approx -\frac{d\mu(Q'_{1/r})}{dr} \Delta r.
\]
Drop the $\Delta r$ on both sides, then it must be an equality, thus except for some $r$'s,
\[\frac{d\mu(P'_r)}{dr} + \frac{d\mu(Q'_{1/r})}{dr}=0.
\]
By Lemma~\ref{lemma_polytope_radial_continuous_lorentzian},
both $\mu(P'_r)$ and $\mu(Q'_{1/r})$ are continuous for $r>0$, so $\mu(P'_r)+\mu(Q'_{1/r})$ is a constant $c$.
As $P$ and $Q$ are in a finite region in $\mathcal{R}$, when $r$ is very big we have $P'_r=P$,
so $Q'_{1/r}$ is the empty set and therefore $c$ is $\mu(P)$.
As $r$ decreases, in $\mathcal{R}$ the surface $x^2=r^2$ sweeps through $P$
and the surface $x^2=1/r^2$ sweeps through $Q$.
When $r$ is close to $0^{+}$ we have $Q'_{1/r}=Q$, 
so $P'_r$ is the empty set and therefore $c$ is also $\mu(Q)$.
So $\mu(P)=\mu(Q)$.

Next we consider the case of $Q=j_{-}(P)$,
which means that both $P$ and $Q$ are inside $\{x\in\mathcal{R}: x^2<0\}$.
For $r>0$, let $P'_{r,-}$ be the intersection of $P$ and $x^2\le -r^2$ in $\mathcal{R}$,
and $Q'_{r,-}$ be the intersection of $Q$ and $x^2\le -r^2$ in $\mathcal{R}$.
By Lemma~\ref{lemma_polytope_radial_continuous_riemannian}
and following essentially the same proof as above, we then prove that $\mu(P)=\mu(Q)$.
\end{proof}

\subsection{Summary of the proof of Theorem~\ref{theorem_mu_volume_invariant_AdS}}

Now we are ready to show that $\mu(P)$ is invariant under \emph{any} isometry $g$ of $\mathcal{R}_D$
(with both $P$ and $g(P)$ in a finite region in $\mathcal{R}$).
As $\mathcal{R}$ is an \emph{open} half-space in $\mathcal{R}_D$,
by (\ref{equation_half_space_R_D}), $g(\mathcal{R})$ can be written as
\begin{equation}
\label{equation_R_isometry}
\{x\in \mathcal{R}_D: h(x) (ax^2+b\cdot x +c) < 0\},
\end{equation}
where $a$ is allowed to be 0,
and $b\cdot x$ is the bilinear product on $\mathbb{R}^{n,1}$ (\ref{equation_bilinear_product_Minkowski}),
with the $x_0$-coordinate of $b$ being 0, and $h(x)$ is defined in (\ref{equation_sign_function_R_D}).

We now summarize the proof of Theorem~\ref{theorem_mu_volume_invariant_AdS} in the following.

\begingroup
\def\thetheorem{\ref{theorem_mu_volume_invariant_AdS}}
\begin{theorem}
For $n\ge 1$, let $P\in\mathcal{H}_0$ in $\mathbb{DH}^{n+1}_1$ be in a finite region in $\mathcal{R}$,
then $\mu(P)$ exists and is invariant under isometries of $\mathcal{R}_D$
(for isometries $g$ with $g(P)$ also in a finite region in $\mathcal{R}$).
\end{theorem}
\addtocounter{theorem}{-1}
\endgroup

\begin{proof}
By Remark~\ref{remark_union_polytopes_AdS}, 
we only need to consider the case that $P$ is a good polytope in a finite region in $\mathcal{R}$.
By Lemma~\ref{lemma_volume_mu_exist}, $\mu(P)$ exists.
For an isometry $g$ of $\mathcal{R}_D$ (with $g(P)$ also in a finite region in $\mathcal{R}$),
we classify $g$ into the following cases by (\ref{equation_R_isometry}).

(1) $a=0$ and $b=0$. This means that $g(\mathcal{R})=\mathcal{R}$.
So by Lemma~\ref{lemma_invariance_preserve_R}, $\mu(g(P))=\mu(P)$. 

(2) $a\ne 0$. Because $\mathcal{R}$ is not a good half-space of $\mathcal{R}_D$ 
(see Remark~\ref{remark_discriminant}),
the face of $g(\mathcal{R})$ in $\mathcal{R}$ has degenerate metric and is a light cone centered on $x_0=0$
(see Figure~\ref{figure_faces_R_D} (c)).
Shift the apex of the light cone to the origin by a translation $s$,
then $s(g(P))$ is in $\{x\in\mathcal{R}: x^2<0\}$ if $a>0$ or in $\{x\in\mathcal{R}: x^2>0\}$ if $a<0$.
Then by an inversion $j_0$ (it is $j$ if $a<0$, or $j_{-}$ if $a>0$),
we have $j_0(s(g(\mathcal{R})))=\mathcal{R}$.
By Lemma~\ref{lemma_invariance_preserve_R}, 
$j_0(s(g(P)))$ is in a finite region in $\mathcal{R}$ and $\mu(j_0(s(g(P))))=\mu(P)$.
As $g(P)$ is in a finite region in $\mathcal{R}$, so is $s(g(P))$.
Then by Lemma~\ref{lemma_mu_invariant_inversion},
$\mu(g(P))=\mu(s(g(P)))=\mu(j_0(s(g(P))))=\mu(P)$.
So $\mu(g(P))=\mu(P)$.

(3) $a=0$ and $b\ne 0$. 
The face of $g(\mathcal{R})$ in $\mathcal{R}$ is a plane that satisfies $b\cdot x+c=0$.
By a translation $s$ along the $x_n$ direction, $g(P)$ can be shifted into the region $x^2<0$ 
such that $s(g(P))$ dose not touch the surface $x^2=0$,
then by inversion $j_{-}$ we have $j_{-}(s(g(P)))$ in a finite region in $\mathcal{R}$.
Thus by Lemma~\ref{lemma_mu_invariant_inversion}, $\mu(g(P))=\mu(s(g(P)))=\mu(j_{-}(s(g(P))))$.
The translation $s$ can be flexible enough such that
the face of $s(g(\mathcal{R}))$ in $\mathcal{R}$ satisfies $b\cdot x+c_0=0$ with $c_0\ne 0$,
so by Proposition~\ref{proposition_half_space_inversion},
the face of $j_{-}(s(g(\mathcal{R})))$ in $\mathcal{R}$ satisfies $-c_0x^2+b\cdot x=0$ with $c_0\ne 0$.
By case (2) we have $\mu(j_{-}(s(g(P))=\mu(P)$, so $\mu(g(P))=\mu(P)$.

By the cases above, $\mu(P)$ is invariant under isometries of $\mathcal{R}_D$.
\end{proof}

\section{Proof of Theorem~\ref{theorem_volume_AdS_real_imaginary}}

With Theorem~\ref{theorem_mu_volume_invariant_AdS} proved,
then it proves Theorem~\ref{theorem_polytope_volume_finite_invariant}
(see Section~\ref{section_proof_volume_invariant} and Definition~\ref{definition_volume_polytope_AdS}),
namely  $V_{n+1}(P)$ is well defined on $\mathcal{H}_0$ in $\mathbb{DH}^{n+1}_1$
(see Definition~\ref{definition_algebra_half_space}).
We remark that while $\mathcal{R}$ and $\mathcal{R}_{-}$  
(see (\ref{equation_Lorentz_metric}) and  (\ref{equation_Lorentz_metric_negative}))
are half-spaces in $\mathbb{DH}^{n+1}_1$,
they are not \emph{good} half-spaces (see Remark~\ref{remark_discriminant}),
so they are not elements of $\mathcal{H}_0$ and may not have a well defined 
$V_{n+1}(\mathcal{R})$ and $V_{n+1}(\mathcal{R}_{-})$.
We next show that $V_{n+1}(P)$ is only finitely but not countably additive.

\begin{example}
\label{example_countably_additive_AdS}
Let $B_1$ in $\mathbb{DH}^{n+1}_1$ be a good polytope in a finite region in $\mathcal{R}$
containing the origin $O$ and with non-zero volume $V_{n+1}(B_1)$.
By proportionally shrinking $B_1$ to the interior of $B_1$, we obtain $B_2$.
Similarly, we construct $B_{i+1}$ from $B_i$ for all $i\ge 1$,
and let $P_i=B_i\setminus B_{i+1}$. As $V_{n+1}(B_i)$ is invariant under similarity in $\mathcal{R}$,
so $V_{n+1}(P_i)=V_{n+1}(B_{i+1})-V_{n+1}(B_i)=0$,
and thus $\sum_{i=1}^{\infty}V_{n+1}(P_i)=0$.
As $V_{n+1}(\bigcup_{i=1}^{\infty}P_i)=V_{n+1}(B_1\setminus\{O\})\ne 0$, therefore
$V_{n+1}(\bigcup_{i=1}^{\infty}P_i)\ne \sum_{i=1}^{\infty}V_{n+1}(P_i)$.
Thus $V_{n+1}(P)$ is not countably additive.
\end{example}

If $P\in\mathcal{H}_0$ in $\mathbb{DH}^{n+1}_1$ is in a finite region in $\mathcal{R}$,
denote by $P_{+}$ the upper portion of $P$ with $x_0>0$,
and $P_{-}$ the lower portion with $x_0<0$ respectively.
If $P$ has finite \emph{standard} volume (see Definition~\ref{definition_finite_standard_volume}),
because $P_{-}$ is the mirror image of $P_{+}$, 
then $V_{n+1}(P_{+})=V_{n+1}(P_{-})$ for $n$ odd,
 thus $V_{n+1}(P)=2V_{n+1}(P_{+})$;
and $V_{n+1}(P_{+})=-V_{n+1}(P_{-})$ for $n$ even,
thus they cancel each other out and therefore $V_{n+1}(P)=0$.
We further have the following result for $n$ even,
a middle step before proving Theorem~\ref{theorem_volume_AdS_real_imaginary}.

\begin{corollary}
\label{corollary_volume_polytope_zero}
For $n$ even and $n\ge 2$, let $P\in\mathcal{H}_0$ in $\mathbb{DH}^{n+1}_1$.
If the intersection of $P$ and $\partial\mathbb{H}^{n+1}_1$ is less than $n$-dimensional (the full dimension), 
then $V_{n+1}(P)=0$.
\end{corollary}

\begin{proof}
By Remark~\ref{remark_union_polytopes_AdS},
we can just consider the case that $P$ is a good polytope in a finite region in $\mathcal{R}$.
We may also assume that $P$ is the intersection of closed half-spaces, 
which does not affect the assumption in the statement.
As the intersection of $P$ and $\partial\mathbb{H}^{n+1}_1$ is less than full dimensional,
similar to polytopes in hyperbolic space,
because all the facets of a good polytope have non-degenerate metrics,
$P$ only intersects $\partial\mathbb{H}^{n+1}_1$ 
at at most a finite number of points (called \emph{ideal vertices}).
Let $E_t$ be the intersection of $P$ and the plane $x_n=t$ in $\mathcal{R}$.
By Lemma~\ref{lemma_volume_mu_exist} (and replacing $\mu(P)$ with $V_{n+1}(P)$ 
as Theorem~\ref{theorem_polytope_volume_finite_invariant} is already proved),
we have $V_{n+1}(P)=\int b(t)dt$ with
\[b(t)=-\frac{1}{n}\sum_{F\subset E_t}\pm\frac{1}{r_F}V_{n-1}(F),
\]
where $F$ are $(n-1)$-faces of $E_t$ with radius $r_F$,
with the plus sign for top faces and the minus sign for bottom faces $F$ of $E_t$
(see Definition~\ref{definition_face_top_bottom_side}).
Except for a finite number of $t$'s where $P$ has the ideal vertices on $\partial\mathbb{H}^{n+1}_1$,
$E_t$ does not intersect (or ``touch'') with $\partial\mathbb{H}^{n+1}_1$.
Then for the upper portion $F_{+}$ and lower portion $F_{-}$,
both $V_{n-1}(F_{+})$ and $V_{n-1}(F_{-})$ are finite,
and $V_{n-1}(F)=V_{n-1}(F_{+})+V_{n-1}(F_{-})=0$.
Therefore $V_{n+1}(P)=\int b(t)dt=0$.
\end{proof}

More generally, for $P\in\mathcal{H}_0$ in $\mathbb{DH}^{n+1}_1$ in a finite region in $\mathcal{R}$,
no matter $P$ has finite standard volume or not, we have
\begin{equation}
\label{equation_mu_epsilon_P_plus}
\mu_{\epsilon}(P)=\mu_{\epsilon}(P_{+})+\mu_{\epsilon}(P_{-})
=\mu_{\epsilon}(P_{+})+(-1)^{n+1} \mu_{-\epsilon}(P_{+}).
\end{equation}
Taking the \emph{pointwise} sum of $\mu_{\epsilon}(P_{+})+(-1)^{n+1}\mu_{-\epsilon}(P_{+})$ on $P_{+}$
(see (\ref{equation_volume_P_half_space_AdS_epsilon})),
then $\mu_{\epsilon}(P)$ is real for $n$ odd and imaginary for $n$ even.
As by definition $V_{n+1}(P)=\mu(P)=\lim_{\epsilon\to 0^{+}} \mu_{\epsilon}(P)$,
so $V_{n+1}(P)$ is also real for $n$ odd and imaginary for $n$ even.
Besides, $\mu_{\epsilon}(P)=\mu_{-\epsilon}(P)$ for $n$ odd, 
and $\mu_{\epsilon}(P)=-\mu_{-\epsilon}(P)$ for $n$ even,
so the choice of the sign of $\epsilon$ affects 
the definition of $V_{n+1}(P)$ for $n$ even, but not for $n$ odd.

We now prove Theorem~\ref{theorem_volume_AdS_real_imaginary}, which we recall below.

\begingroup
\def\thetheorem{\ref{theorem_volume_AdS_real_imaginary}}
\begin{theorem}
Let $P\in\mathcal{H}_0$ in $\mathbb{DH}^{n+1}_1$,
then $V_{n+1}(P)$ is real for $n$ odd,
and $V_{n+1}(P)$ is imaginary for $n$ even and is completely determined by 
the intersection of $P$ and $\partial\mathbb{H}^{n+1}_1$.
\end{theorem}
\addtocounter{theorem}{-1}
\endgroup

\begin{proof}
By the argument above, the only thing left is to show that for $n$ even,
$V_{n+1}(P)$ is completely determined by the intersection of $P$ and $\partial\mathbb{H}^{n+1}_1$.
Assume $P'\in\mathcal{H}_0$ and $P\cap \partial\mathbb{H}^{n+1}_1=P'\cap \partial\mathbb{H}^{n+1}_1$,
then the intersection of $P\setminus P'$ and $\partial\mathbb{H}^{n+1}_1$ is an empty set.
By Corollary~\ref{corollary_volume_polytope_zero},
we have $V_{n+1}(P\setminus P')=0$. Therefore
\[V_{n+1}(P)=V_{n+1}(P\setminus P') + V_{n+1}(P\cap P')
=V_{n+1}(P\cap P').
\]
By symmetry we also have 
$V_{n+1}(P')=V_{n+1}(P\cap P')$, so $V_{n+1}(P)=V_{n+1}(P')$.
\end{proof}

\section{A Schl\"{a}fli differential formula for $\mathbb{DH}^{n+1}_1$}

We next obtain a Schl\"{a}fli differential formula for $\mathbb{DH}^{n+1}_1$,
which will be helpful for introducing the corresponding theories on $\partial\mathbb{H}^{n+1}_1$
(Section~\ref{section_volume_boundary}), but it is also of interest in its own right.
See Milnor~\cite{Milnor:Schlafli} for the background of the formula,
and see also Rivin and Schlenker~\cite{RivinSchlenker}, Su{\'a}rez-Peir{\'o}~\cite{Suarez:deSitter}
and Zhang~\cite{Zhang:rigidity} for some generalizations.
A Schl\"{a}fli differential formula for $\mathbb{DH}^n$ was also obtained
\cite[Theorem~1.2]{Zhang:double_hyperbolic}.
The formula relates the change of the volume of a polytope to the change of its dihedral angles
(at the codimension 2 faces).
But we note that, unlike in Riemannian geometry where an angle is in general uniquely defined,
in Minkowski space the definition of an angle is not unique and may depend on the context it is addressing.
For different treatments of the notion of angles in Minkowski space, e.g., see Alexandrov~\cite{Alexandrov:Minkowski},
Cho and Kim~\cite{ChoKim}, Schlenker~\cite{Schlenker:cross} and Su{\'a}rez-Peir{\'o}~\cite{Suarez:deSitter}.
Our definition agrees with \cite{Alexandrov:Minkowski}, but uses a slightly different approach.

For the purpose of this paper, we are only interested in defining a dihedral angle between two facets
(with non-degenerate metrics) whose intersection is a \emph{Riemannian} codimension 2 face.
Then the angle can be defined in a Minkowski 2-plane between two \emph{non}-null vectors.
To do so, we will use the $x_0x_n$-plane as example,
and define an angle $\theta_A$ for a triangle $ABC$ whose sides are not null vectors.
The \emph{length} $c$ of edge $AB$ is defined by $(\overrightarrow{AB}^2)^{1/2}$, 
which is a positive real number if $AB$ is in a spacelike direction,
and is the product of a positive real number by the imaginary unit $i$ if $AB$ is in a timelike direction;
and the same for length $b$ of $AC$ and length $a$ of $BC$.
As the sides are not null vectors, we have $abc\ne 0$.
The \emph{area} of triangle $ABC$ is defined by integrating the area element $dx_0dx_n$ over the region,
which is always a positive real number.

The \emph{angle} $\theta_A$ at $A$ is defined such that
\begin{equation}
\label{equation_triangle_angle_area}
bc \sinh\theta_A=2\cdot\text{area}(ABC),
\end{equation}
where the coefficient 2 is because the \emph{parallelogram} with sides $AB$ and $AC$
has two times the area of the triangle $ABC$.
In the $x_0x_n$-plane, the light cone centered at $A$ cuts the plane into four parts 
(see the dashed lines in Figure~\ref{figure_angle_Minkowski} (a)).
If $AB$ and $AC$ are in the same part and in spacelike directions (e.g., in the $x_0$-direction), 
by (\ref{equation_triangle_angle_area}), as $bc$ is positive so $\sinh\theta_A$ is positive 
and then $\theta_A$ is defined as a positive real number.
Similarly, if $AB$ and $AC$ are in the same part and in timelike directions
(e.g., in the $x_n$-direction), by (\ref{equation_triangle_angle_area}), as $bc$ is negative so $\sinh\theta_A$ is negative
and then $\theta_A$ is defined as a negative real number.
If $AB$ and $AC$ are in a spacelike and a timelike direction respectively,
and $2\cdot\text{area}(ABC) = |bc|$, then we call $\theta_A$ a \emph{right angle} in the Minkowski 2-plane.
As $b$ and $c$ contain one real number and one imaginary number,
so $i\sinh\theta_A=1$ and $\sinh\theta_A=-i$, then we define a right angle to be
\begin{equation}
\label{equation_right_angle}
\theta_A=-\frac{\pi}{2}i. 
\end{equation}
We remark that a right angle in the $x_0x_n$-plane can be completely inside another right angle.
Those information is enough to determine the angle $\theta_A$ for all other cases.

\begin{figure}[h]
\centering
\resizebox{.5\textwidth}{!}
  {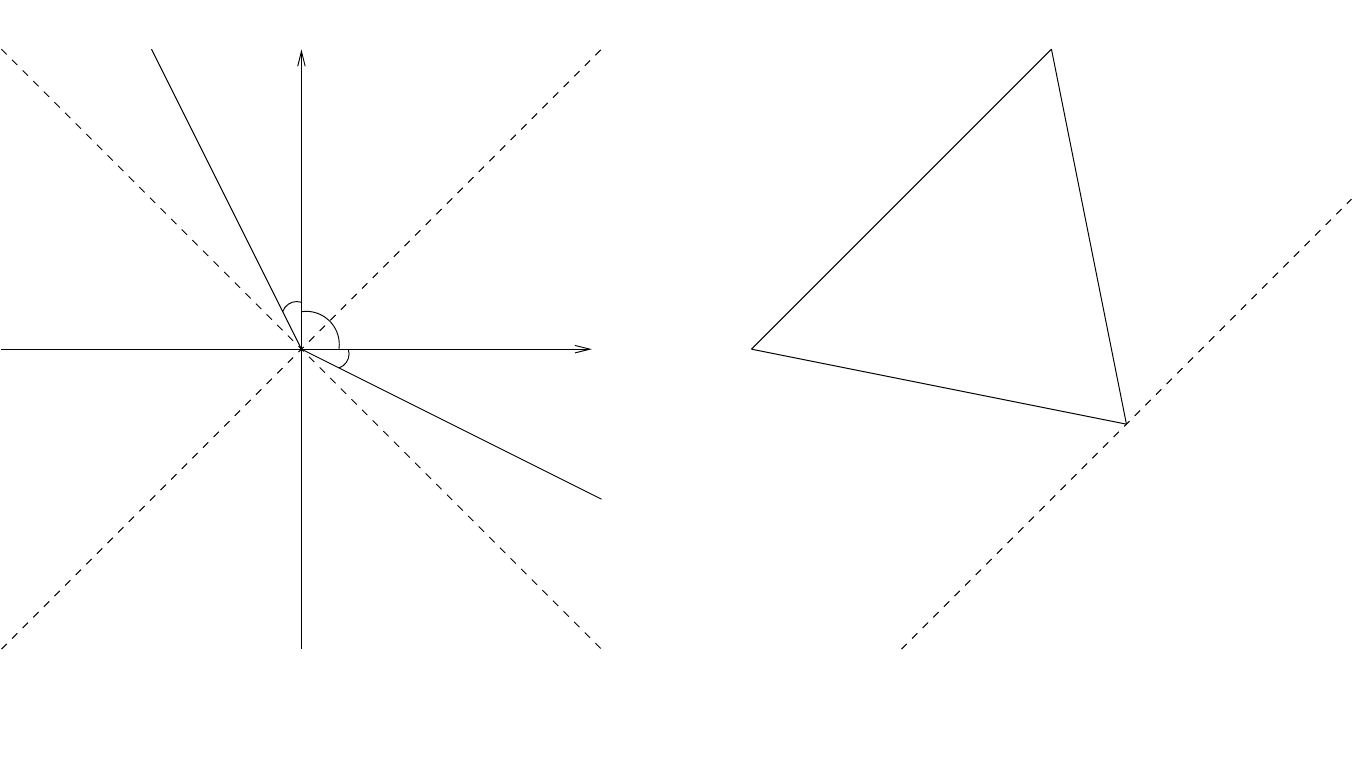}
\caption{The angle $\theta_A$ at $A$}
\label{figure_angle_Minkowski}
\end{figure}

A \emph{straight angle} $\theta$ in a Minkowski 2-plane is the angle between any two non-null vectors pointing to opposite directions.
It is the sum of two right angles, so a straight angle is
\begin{equation}
\label{equation_straight_angle}
\theta=-\pi i. 
\end{equation}
By an elementary geometry property, 
the interior angles of the triangle add up to a straight angle (see Figure~\ref{figure_angle_Minkowski} (b)),
therefore 
\begin{equation}
\label{equation_triangle_angle_sum}
\theta_A+\theta_B+\theta_C=-\pi i.
\end{equation}
On both sides of (\ref{equation_triangle_angle_area}) divide by $abc$,
as the right side is symmetric to all edges, so 
\begin{equation}
\label{equation_angle_length_ratio}
\frac{\sinh\theta_A}{a}=\frac{\sinh\theta_B}{b}=\frac{\sinh\theta_C}{c}.
\end{equation}

We have the following special case of Schl\"{a}fli differential formula for $\mathbb{DH}^{n+1}_1$,
and though not a ``full version'', it is strong enough for our purpose.

\begin{lemma}
\label{lemma_mu_SDF_special_AdS}
For $n\ge 1$, let $P$ in $\mathbb{DH}^{n+1}_1$
be a good polytope in a finite region in $\mathcal{R}$,
and $P_t$ be the intersection of $P$ and $x_n\le t$, with $E_t$ the $n$-face of $P_t$ on $x_n=t$.
For each $(n-1)$-face $F$ of $E_t$, let $\theta_F$ be the dihedral angle at $F$
and $r_F$ be the radius of $F$. 
Then $V_{n+1}(P_t)$ is continuous for $t$, and for $\kappa=-1$,
\[\kappa\cdot \frac{dV_{n+1}(P_t)}{dt}
=\frac{1}{n}\sum_{F\subset E_t}V_{n-1}(F)\frac{d\theta_F}{dt}.
\]
\end{lemma}

\begin{proof}
By Lemma~\ref{lemma_volume_mu_exist} (and replacing $\mu(P_t)$ with $V_{n+1}(P_t)$ 
as Theorem~\ref{theorem_polytope_volume_finite_invariant} is already proved),
$V_{n+1}(P_t)$ is continuous for $t$ and 
\begin{equation}
\label{equation_volume_SDF_special_AdS_radius}
\frac{dV_{n+1}(P_t)}{dt}=-\frac{1}{n}\sum_{F\subset E_t}\pm\frac{1}{r_F}V_{n-1}(F),
\end{equation}
with the plus sign for top faces and the minus sign for bottom faces $F$ of $E_t$
(see Definition~\ref{definition_face_top_bottom_side}).
The side faces $F$ of $E_t$, whose radius $r_F$ can be treated as $\infty$, has $\frac{1}{r_F}=0$.
Now all we need to do is to show that $\frac{d\theta_F}{dt}=\pm \frac{1}{r_F}$.

We first consider the case that $F$ is the intersection of $E_t$ and a top $n$-face of $P_t$.
If the top $n$-face is Lorentzian, then it is on a surface $(x-v)^2=r^2$ in $\mathcal{R}$ with $r>0$,
where $v$ is a vector in the plane $x_0=0$ with $c$ the $x_n$-coordinate of $v$.
For the triangle in Figure~\ref{figure_angle_triangle} (a), 
by using (\ref{equation_right_angle}) and (\ref{equation_triangle_angle_sum}),
the lengths are $r$, $r_F$, $(t-c)i$, and the angles are $-\frac{\pi}{2}i$, $-\theta_F-\pi i$, $\theta_F+\frac{\pi}{2}i$ respectively.
Then by (\ref{equation_angle_length_ratio}), we have
$\frac{\sinh\left(-\frac{\pi}{2}i\right)}{r}=\frac{\sinh(-\theta_F-\pi i)}{r_F}=\frac{\sinh\left(\theta_F+\frac{\pi}{2}i\right)}{(t-c)i}$.
So
\[r\cdot\sinh(-\theta_F-\pi i)=r_F\cdot\sinh\left(-\frac{\pi}{2}i\right),
\quad
r\cdot\sinh\left(\theta_F+\frac{\pi}{2}i\right)=(t-c)i\cdot\sinh\left(-\frac{\pi}{2}i\right).
\]
Differentiating the right hand equation with respect to $t$, we have $ri\cdot\sinh\theta_F\frac{d\theta_F}{dt}=1$.
By the left hand equation we have $r\cdot\sinh\theta_F=-r_Fi$, then $\frac{d\theta_F}{dt}=\frac{1}{r_F}$.

If the top $n$-face of $P_t$ is Riemannian, then it is on a surface $(x-v)^2=-r^2$ in $\mathcal{R}$ with $r>0$,
where $v$ is a vector in the plane $x_0=0$ with $c$ the $x_n$-coordinate of $v$.
For the triangle in Figure~\ref{figure_angle_triangle} (b),
by using (\ref{equation_right_angle}) and (\ref{equation_triangle_angle_sum}) again,
the lengths are $ri$, $r_F$, $(t-c)i$, and the angles are $-\frac{\pi}{2}i$, $-\theta_F$, $\theta_F-\frac{\pi}{2}i$ respectively.
Then by (\ref{equation_angle_length_ratio}) again, we have
$\frac{\sinh\left(-\frac{\pi}{2}i\right)}{ri}=\frac{\sinh(-\theta_F)}{r_F}=\frac{\sinh\left(\theta_F-\frac{\pi}{2}i\right)}{(t-c)i}$.
So
\[ri\cdot\sinh(-\theta_F)=r_F\cdot\sinh\left(-\frac{\pi}{2}i\right),
\quad
ri\cdot\sinh\left(\theta_F-\frac{\pi}{2}i\right)=(t-c)i\cdot\sinh\left(-\frac{\pi}{2}i\right).
\]
Differentiating the right hand equation with respect to $t$, we have $r\cdot\sinh\theta_F\frac{d\theta_F}{dt}=1$.
By the left hand equation we have $-ri\cdot\sinh\theta_F=-r_Fi$, then $\frac{d\theta_F}{dt}=\frac{1}{r_F}$.

Next, if $F$ is the intersection of $E_t$ and a bottom $n$-face of $P_t$, 
similarly $\frac{d\theta_F}{dt}=-\frac{1}{r_F}$. 
If $F$ is the intersection of $E_t$ and a side $n$-face of $P_t$, as $\theta_F$ does not change,
so $\frac{d\theta_F}{dt}=0$.

In (\ref{equation_volume_SDF_special_AdS_radius}), 
replacing $\pm \frac{1}{r_F}$ with $\frac{d\theta_F}{dt}$,
and multiplying $\kappa=-1$ on both sides, we then finish the proof.
\end{proof}

\begin{figure}[h]
\centering
\resizebox{.6\textwidth}{!}
  {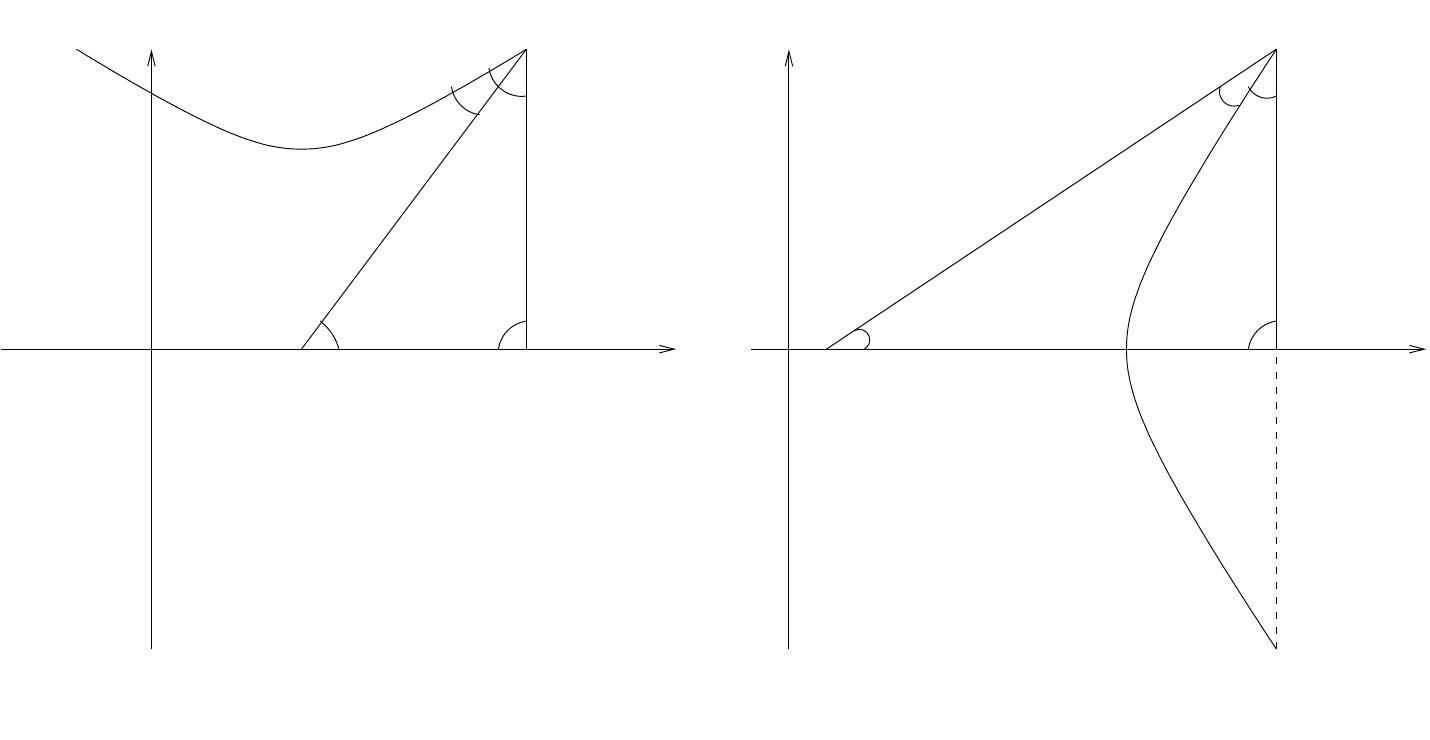}
\caption{The dihedral angle $\theta_F$ at the intersection of $x_n=t$ and a top face}
\label{figure_angle_triangle}
\end{figure}

\section{A volume on $\partial\mathbb{H}^{n+1}_1$ for $n$ even}
\label{section_volume_boundary}

An important application of Theorem~\ref{theorem_volume_AdS_real_imaginary}
is that for $n$ even and $n\ge 2$,
the volume on $\mathbb{DH}^{n+1}_1$ induces an intrinsic \emph{volume} on $\partial\mathbb{H}^{n+1}_1$
that is invariant under conformal transformations of $\partial\mathbb{H}^{n+1}_1$,
namely, invariant under isometries of $\mathbb{H}^{n+1}_1$.
For a similar notion of volume on $\partial\mathbb{H}^{n+1}$ for $n$ even, see Section~\ref{section_background}.
We first clarify some notions for all $n\ge 1$.
Recall that a \emph{good half-space} in $\mathbb{DH}^{n+1}_1$ is a half-space
whose face has non-degenerate metric (see Definition~\ref{definition_algebra_half_space}).

\begin{definition}
For a half-space (resp. good half-space) in $\mathbb{DH}^{n+1}_1$,
we call its restriction to $\partial\mathbb{H}^{n+1}_1$ a \emph{half-space} 
(resp. \emph{good half-space}) in $\partial\mathbb{H}^{n+1}_1$.
A \emph{polytope} (resp. \emph{good polytope}) in $\partial\mathbb{H}^{n+1}_1$ is a finite intersection of 
half-spaces (resp. good half-spaces) in $\partial\mathbb{H}^{n+1}_1$.
Let $\mathcal{F}$ (resp. $\mathcal{F}_0$) be the algebra over $\partial\mathbb{H}^{n+1}_1$
generated by half-spaces (resp. good half-spaces) in $\partial\mathbb{H}^{n+1}_1$.
\end{definition}

By definition, a polytope (resp. good polytope) in $\partial\mathbb{H}^{n+1}_1$
can also be viewed as a restriction of a polytope (resp. good polytope) $P$ 
in $\mathbb{DH}^{n+1}_1$ to $\partial\mathbb{H}^{n+1}_1$.
We remark that the choice of $P$ in $\mathbb{DH}^{n+1}_1$ may not be unique,
but this is not a concern for our results.

For $n$ even and $n\ge 2$, for any $G\in\mathcal{F}_0$ in $\partial\mathbb{H}^{n+1}_1$,
by definition there is $P\in\mathcal{H}_0$ in $\mathbb{DH}^{n+1}_1$ 
(see Definition~\ref{definition_algebra_half_space}),
such that $G=P\cap \partial\mathbb{H}^{n+1}_1$.
On $\partial\mathbb{H}^{n+1}_1$ we define a real valued \emph{volume} 
$V_{\infty,n}(G)$ of $G$ by
\begin{equation}
\label{equation_volume_boundary_AdS}
V_{\infty,n}(G) := c_n\cdot V_{n+1}(P),
\quad\text{where}\quad
c_n=\frac{V_n(\mathbb{S}^n)}{i^{n+1}V_{n+1}(\mathbb{S}^{n+1})},
\end{equation}
with $V_n(\mathbb{S}^n)$ the volume of the standard unit $n$-sphere $\mathbb{S}^n$.
By Theorem~\ref{theorem_volume_AdS_real_imaginary}, $V_{\infty,n}(G)$ is well defined for $n$ even.
As the conformal transformations of $\partial\mathbb{H}^{n+1}_1$
are induced by the isometries of $\mathbb{H}^{n+1}_1$,
so combined with Theorem~\ref{theorem_polytope_volume_finite_invariant}, 
it immediately proves Theorem~\ref{theorem_volume_boundary_at_infinity}.

\begingroup
\def\thetheorem{\ref{theorem_volume_boundary_at_infinity}}
\begin{theorem}
For $n$ even and $n\ge 2$, let $G\in\mathcal{F}_0$ in $\partial\mathbb{H}^{n+1}_1$,
then $V_{\infty,n}(G)$ is well defined and invariant under conformal transformations
of $\partial\mathbb{H}^{n+1}_1$.
\end{theorem}
\addtocounter{theorem}{-1}
\endgroup

\begin{remark}
For $n$ even and $n\ge 2$, let $P$ be a good polytope in a finite region in $\mathcal{R}$ 
(see (\ref{equation_Lorentz_metric})) containing the origin and with non-zero volume $V_{n+1}(P)$.
Let $G=P\cap \partial\mathbb{H}^{n+1}_1$,
then $V_{\infty,n}(G)$ is non-zero and is invariant under similarity.
So $G$ can be proportionally shrunk to arbitrarily \emph{small} size 
while keeping $V_{\infty,n}(G)$ a fixed non-zero volume.
This implies that $V_{\infty,n}(G)$
is not countably additive on $\mathcal{F}_0$ (see Example~\ref{example_countably_additive_AdS}),
and is not induced by any volume form on $\partial\mathbb{H}^{n+1}_1$ as a differentiable manifold.
Besides, $V_{\infty,n}(G)$ also takes values positive, negative and zero as well.
\end{remark}

\begin{remark}
For $n$ even, if $P$ is a polytope in the double hyperbolic space $\mathbb{DH}^{n+1}$ and
$G=P\cap \partial\mathbb{H}^{n+1}$, then on $\partial\mathbb{H}^{n+1}$,
in \cite[Theorem~12.1]{Zhang:double_hyperbolic} $V_{\infty,n}(G)$ was defined the same way as in 
(\ref{equation_volume_boundary_AdS}) with the same factors $c_n$,
and is well defined and invariant under M\"{o}bius transformations of $\partial\mathbb{H}^{n+1}$.
The factors $c_n$ for $\partial\mathbb{H}^{n+1}$ were chosen in a way such that
$V_{\infty,n}(\partial\mathbb{H}^{n+1}) = V_n(\mathbb{S}^n)$, 
but we remark that the choices are not unique, e.g.,
where they may also be chosen such that $V_{\infty,n}(\partial\mathbb{H}^{n+1}) = V_n(\mathbb{DH}^n)$ instead.
\end{remark}

To our knowledge, for $n$ even, both the definition of $V_{\infty,n}(G)$ on $G$ in $\partial\mathbb{H}^{n+1}_1$ 
and its conformal invariance property are new. We ask the following question.

\begin{question}
For $n$ even and $n\ge 2$, can the volume $V_{\infty,n}(G)$ on $\partial\mathbb{H}^{n+1}_1$ 
be defined for a larger class of regions than the algebra $\mathcal{F}_0$ of $\partial\mathbb{H}^{n+1}_1$?
\end{question}

\begin{remark}
\label{remark_volume_infinity_2}
In $\partial\mathbb{H}^3_1$,
for any region $U$ with piecewise smooth boundary that is nowhere ``tangent'' to a null line,
a \emph{potential} definition of $V_{\infty,2}(U)$ is as follows.
First if $U$ is homeomorphic to a closed disk, 
assume the dihedral angles between consecutive sides are $\theta_i$,
then define $V_{\infty,2}(U):=\operatorname{Re} \sum_{i}\theta_i$.
It can be shown that the volume is well defined and additive.
Simply from its definition, $V_{\infty,2}(U)$ 
is invariant under conformal transformations of $\partial\mathbb{H}^3_1$.
\end{remark}

For $n=2m$, we now obtain a special case of Schl\"{a}fli differential formula for $\partial\mathbb{H}^{2m+1}_1$,
which is also a middle step to prove a more import result 
Theorem~\ref{theorem_polytope_volume_infinity_AdS} later.

\begin{lemma}
\label{lemma_SDF_special_boundary_AdS}
For $m\ge 1$, let $G$ in $\partial\mathbb{H}^{2m+1}_1$ 
be a good polytope in a finite region in $x_0=0$ in $\mathcal{R}$,
and $G_t$ be the intersection of $G$ and $x_{2m}\le t$. 
Then $V_{\infty,2m}(G_t)$ is continuous for $t$, and
\begin{equation*}
\frac{dV_{\infty,2m}(G_t)}{dt}
=\frac{1}{2m-1}\sum_{H} V_{\infty,2m-2}(H)\frac{d\theta_H}{dt},
\end{equation*}
where the sum is taken over all $(2m-2)$-faces $H$ of $G_t$ on $x_{2m}=t$,
with $\theta_H$ the dihedral angle at $H$.
For $2m-2=0$, $V_{\infty,0}(H)$ is the number of points in $H$.
\end{lemma}

\begin{proof}
By definition there is a good polytope $P$ in $\mathbb{DH}^{2m+1}_1$ 
such that $G=P\cap \partial\mathbb{H}^{2m+1}_1$.
Without loss of generality, we assume $P$ is in a finite region in $\mathcal{R}$ 
(e.g., by taking intersections with other half-spaces of $\mathbb{DH}^{2m+1}_1$ if needed).
Let $P_t$ be the intersection of $P$ and $x_{2m}\le t$ in $\mathcal{R}$,
with $E_t$ the $(2m)$-face of $P_t$ on $x_{2m}=t$.
By Lemma~\ref{lemma_mu_SDF_special_AdS},
$V_{2m+1}(P_t)$ is continuous for $t$, and for $\kappa=-1$,
\begin{equation}
\label{equation_SDF_special_AdS}
\kappa\cdot \frac{dV_{2m+1}(P_t)}{dt}
=\frac{1}{2m}\sum_{F\subset E_t} V_{2m-1}(F)\frac{d\theta_F}{dt},
\end{equation}
where the sum is taken over all $(2m-1)$-faces $F$ of $E_t$ on $x_{2m}=t$,
with $\theta_F$ the dihedral angle at $F$.
For each $F\subset E_t$, if it does not intersect (or ``touch'') with $x_0=0$, 
then for the upper portion $F_{+}$ and lower portion $F_{-}$,
both $V_{2m-1}(F_{+})$ and $V_{2m-1}(F_{-})$ are finite,
so $V_{2m-1}(F)=V_{2m-1}(F_{+})+V_{2m-1}(F_{-})=0$
and does not contribute to (\ref{equation_SDF_special_AdS}).
Otherwise let $H$ be the intersection of $F$ and $x_0=0$. As
\[V_{\infty,2m}(G_t)=c_{2m}\cdot V_{2m+1}(P_t),
\quad V_{\infty,2m-2}(H)=c_{2m-2}\cdot V_{2m-1}(F),
\quad\theta_H=\theta_F,
\]
by plugging them into (\ref{equation_SDF_special_AdS}) 
(the only non-trivial computation is to use a well known recursive formula 
$V_n(\mathbb{S}^n)=\frac{2\pi}{n-1} V_{n-2}(\mathbb{S}^{n-2})$),
we then finish the proof.
\end{proof}

We have the following result for $\partial\mathbb{H}^3_1$.
See also Remark~\ref{remark_volume_infinity_2}.

\begin{corollary}
\label{corollary_volume_infinity_2}
Let $G$ in $\partial\mathbb{H}^3_1$ be a good polytope in a finite region in $x_0=0$ in $\mathcal{R}$.
If $G$ is homeomorphic to a closed disk and has $k$ sides with angles $\theta_i$ between them, then
\begin{equation*}
V_{\infty,2}(G)=\sum_{i}\theta_i+(k-2)\pi i=\operatorname{Re} \sum_{i}\theta_i.
\end{equation*}
\end{corollary}

\begin{proof}
By Lemma~\ref{lemma_SDF_special_boundary_AdS}, with the details skipped, 
we can verify that $V_{\infty,2}(G)$ has the form $\sum_{i}\theta_i+c$,
and by the fact that any straight angle in the plane $x_0=0$ 
is $-\pi i$ (\ref{equation_straight_angle}), we have $c=(k-2)\pi i$.
As $V_{\infty,2}(G)$ is real, so $V_{\infty,2}(G)=\operatorname{Re} \sum_{i}\theta_i$.
\end{proof}

\begin{example}
Let $G$ in $\partial\mathbb{H}^3_1$ be a good polytope with flat edges in a finite region in $x_0=0$ in $\mathcal{R}$
(like a polygon in a Minkowski 2-plane).
If $G$ is a triangle, then by Corollary~\ref{corollary_volume_infinity_2} we have
$V_{\infty,2}(G)=\sum_{i}\theta_i+\pi i=0$ (see (\ref{equation_triangle_angle_sum})).
Otherwise, we can cut $G$ into triangles so that we still have $V_{\infty,2}(G)=0$.
\end{example}

We have the following generalization for higher dimensional good polytope in $\partial\mathbb{H}^{2m+1}_1$
with \emph{flat} facets in $x_0=0$ in $\mathcal{R}$ (like a convex polytope in a Minkowski $(2m)$-space).

\begin{proposition}
\label{proposition_Minkowski_polytope}
For $m\ge 1$, let $G$ in $\partial\mathbb{H}^{2m+1}_1$ be a good polytope with flat facets 
in a finite region in $x_0=0$ in $\mathcal{R}$, then $V_{\infty,2m}(G)=0$. 
\end{proposition}

\begin{proof}
Let $G_t$ be the intersection of $G$ and $x_{2m}\le t$ in $x_0=0$ in $\mathcal{R}$.
Denote $H$ a $(2m-2)$-face of $G_t$ on $x_{2m}=t$, and $\theta_H$ the dihedral angle at $H$.
Because $G$ has flat facets, so all $\theta_H$ are constants, thus $\frac{d\theta_H}{dt}=0$.
Then by Lemma~\ref{lemma_SDF_special_boundary_AdS}, $V_{\infty,2m}(G_t)$ is continuous for $t$, and
$\frac{dV_{\infty,2m}(G_t)}{dt}=0$.
Since $V_{\infty,2m}(G_t)$ is 0 when $t\to -\infty$, so $V_{\infty,2m}(G)=0$. 
\end{proof}

We remark that if $G$ has all flat facets in $x_0=0$ in $\mathcal{R}_D$ (see Definition~\ref{definition_model_R_D})
but is \emph{not} in a finite region in $x_0=0$ in $\mathcal{R}$
(like an \emph{unbounded} polytope in Minkowski space,
and some part of $G$ may fall in $x_0=0$ in $\mathcal{R}_{-}$), then we do not have $V_{\infty,2m}(G)=0$. 
Proposition~\ref{proposition_Minkowski_polytope} is analogous to (\ref{equation_polytope_volume_infinity}) 
for polytopes in the Euclidean $(2m)$-space with $\kappa=0$.

In fact, the analogy to (\ref{equation_polytope_volume_infinity}) goes further for $\kappa=-1$ as well.
Notice that $\mathbb{DH}^{2m}_1$ (with $\kappa=-1$) is naturally endowed with 
the same conformal structure as $\partial\mathbb{H}^{2m+1}_1$,
which is analogous to the fact that $\mathbb{DH}^{2m}$ (with $\kappa=-1$) is naturally endowed with 
the same conformal structure as $\partial\mathbb{H}^{2m+1}$. 
With a slight abuse of notation, the algebra $\mathcal{H}_0$ of $\mathbb{DH}^{2m}_1$ is a \emph{subalgebra}
of the algebra $\mathcal{F}_0$ of $\partial\mathbb{H}^{2m+1}_1$ 
when $\mathbb{DH}^{2m}_1$ is treated conformally as $\partial\mathbb{H}^{2m+1}_1$.
So for $P\in \mathcal{H}_0$ in $\mathbb{DH}^{2m}_1$, besides its volume $V_{2m}(P)$,
by the algebra $\mathcal{F}_0$ of $\partial\mathbb{H}^{2m+1}_1$,
we can also assign $P$ a ``conformal volume'' $V_{\infty,2m}(P)$.
We have the following important property for $\mathbb{DH}^{2m}_1$.

\begin{theorem}
\label{theorem_polytope_volume_infinity_AdS}
Let $P$ be a good polytope in $\mathbb{DH}^{2m}_1$ with $\kappa=-1$,
then $V_{\infty,2m}(P)=\kappa^m V_{2m}(P)$.
\end{theorem}

\begin{proof}
By Remark~\ref{remark_union_polytopes_AdS},
we may assume that $P$ is a good polytope in a finite region in $\mathcal{R}$..
Let $P_t$ be the intersection of $P$ and $x_{2m-1}\le t$,
with $E_t$ the $(2m-1)$-face of $P_t$ on $x_{2m-1}=t$.
Denote $F$ a $(2m-2)$-face of $E_t$ on $x_{2m-1}=t$,
and $\theta_F$ the dihedral angle at $F$.
By Lemma~\ref{lemma_mu_SDF_special_AdS} and \ref{lemma_SDF_special_boundary_AdS} respectively,
we obtain both $\frac{dV_{2m}(P_t)}{dt}$ and $\frac{dV_{\infty,2m}(P_t)}{dt}$ 
as weighted sums of all $\frac{d\theta_F}{dt}$ with coefficients
$\frac{V_{2m-2}(F)}{\kappa (2m-1)}$ and $\frac{V_{\infty,2m-2}(F)}{2m-1}$ respectively.
By (\ref{equation_polytope_volume_infinity}) we have
$\frac{V_{\infty,2m-2}(F)}{2m-1}=\frac{\kappa^{m-1}V_{2m-2}(F)}{2m-1}
=\kappa^m \frac{V_{2m-2}(F)}{\kappa (2m-1)}$.
As both $V_{2m}(P_t)$ and $V_{\infty,2m}(P_t)$ are continuous for $t$ and are 0 when $t\to -\infty$, so
$V_{\infty,2m}(P)=\kappa^m V_{2m}(P)$.
\end{proof}
 
We caution that Theorem~\ref{theorem_polytope_volume_infinity_AdS} does not hold in general 
when $P$ is not a good polytope in $\mathbb{DH}^{2m}_1$,
because when $\mathbb{DH}^{2m}_1$ is treated conformally as $\partial\mathbb{H}^{2m+1}_1$,
$V_{\infty,2m}(P)$ is invariant under conformal transformations of $\partial\mathbb{H}^{2m+1}_1$,
while $V_{2m}(P)$ is not.


\section*{Acknowledgements}

I would like to thank Wei Luo for many helpful suggestions and discussions.

\bibliographystyle{abbrv}  
\bibliography{AdS_arxiv}   

%
%

\end{document}